\documentclass{IEEEtran}
\usepackage{cite}
\usepackage{amsmath}
\usepackage{amssymb}
\usepackage{amsfonts}
\usepackage{amsthm}
\usepackage{algorithmic}
\usepackage{graphicx}
\usepackage[T1]{fontenc}                    % For using fonts
\usepackage[utf8]{inputenc}                 % For using fonts
\def\BibTeX{{\rm B\kern-.05em{\sc i\kern-.025em b}\kern-.08em
    T\kern-.1667em\lower.7ex\hbox{E}\kern-.125emX}}

\usepackage[dvipsnames]{xcolor} 
\usepackage{colortbl}
\usepackage{mathrsfs}
\usepackage{diagbox}
\usepackage[colorlinks = true,
linkcolor = blue,
urlcolor  = blue,
citecolor = blue]{hyperref}

\usepackage[capitalise, noabbrev, nameinlink]{cleveref}

\newcommand{\ldef}{:=}
\newcommand{\rdef}{=:}
\newcommand{\Mcal}[1]{\mathcal{#1}}
\newcommand{\Mc}[1]{\mathcal{#1}}
\newcommand{\tth}{^{\text{th}}}
\newcommand{\bld}[1]{\mathbf{#1}}

\newtheorem{theorem}{Theorem}[section]
\newtheorem{corollary}[theorem]{Corollary}
\newtheorem{lemma}[theorem]{Lemma}
\newtheorem{example}[theorem]{Example}

\newtheorem{remark}[theorem]{Remark}
\newtheorem{define}[theorem]{Definition}
\newtheorem{problem}[theorem]{Problem}
\newtheorem{assume}[theorem]{Assumption}

\newcommand{\integer}{\ensuremath{\mathbb{Z}}}

\newcommand{\real}{\ensuremath{\mathbb{R}}}

\newcommand{\realnonnegative}{\ensuremath{\mathbb{R}}_{\ge 0}}

\newcommand{\bulletsym}{\hbox{$\bullet$}}
\newcommand{\bulletend}{\relax\ifmmode\else\unskip\hfill\fi\bulletsym}
\newcommand{\squaresym}{\hbox{$\blacksquare$}}
\newcommand{\proofend}{\relax\ifmmode\else\unskip\hfill\fi\squaresym}
\newcommand{\trianglesym}{\hbox{$\blacktriangle$}}
\newcommand{\egend}{\relax\ifmmode\else\unskip\hfill\fi\trianglesym}

\renewenvironment{proof}{\textit{Proof.} }{\proofend}

\renewcommand{\hat}[1]{\widehat{#1}}
\renewcommand{\bar}[1]{\overline{#1}}

\def \agt{\Mcal{A}}
\def \alloc{\Mcal{V}}
\def \alloch{\hat{\alloc}}

\def \diam{\mathrm{diam}}
\def \dg{\mathrm{d}}
\def \edg{\Mcal{E}}

\def \exp{\mathrm{exp}}
\def \eqpt{\Mcal{W}}

\def \fseq{z}
\def \game{\mathscr{G}}

\def \grph{\Mcal{G}}

\def \I{\Mcal{I}}

\def \integer{\mathbb{Z}}
\def \intpos{\integer_{\geq 0}}
\def \kcell{[0,1]}

\def \M{\bld{M}}
\def \ne{\Mc{NE}}
\def \neigh{\Mc{N}}
\def \neighb{\bar{\Mc{N}}}

\def \opt{\Mc{O}}

\def \real{\mathbb{R}}
\def \rf{r}
\def \S{\Mc{S}}
\def \salloc{\Mc{P}}
\def \set{\Mc{S}}
\def \setfunc{C}

\def \supp{\mathrm{tsupp}\,}
\def \switch{\sigma_\mathrm{sw}}
\def \tsk{\Mc{Q}}

\def \unif{\mathrm{Unif}}
\def \v{\bld{v}}
\def \w{\bld{w}}
\def \W{\bld{W}}
\def \wb{\overline{\w}}
\def \Wb{\bar{\W}}
\def \wh{\hat{\w}}

\def \wbar{\overline{w}}
\def \what{\hat{w}}

\def \x{\bld{x}}
\def \X{\Mcal{X}}

\def \y{\bld{y}}
\def \zero{\bld{0}}

\DeclareMathOperator*{\submax}{max^{(2)}}
\DeclareMathOperator*{\argmax}{argmax}
\DeclareMathOperator*{\argsubmax}{argmax^{(2)}}
\renewcommand{\mod}{\,\,\mathrm{mod}\,\,}

\def \dynacr{PBRAG}

\newcommand{\until}[1]{\{1,\cdots, #1\}}
\newcommand{\tskind}[1]{\mathfrak{#1}}

\graphicspath{{./figs/}}

\begin{document}
\title{Distributed Task Allocation for Self-Interested Agents with Partially Unknown Rewards}
\author{Nirabhra Mandal, \IEEEmembership{Student Member, IEEE}, \hspace{0.1ex} Mohammad Khajenejad, \IEEEmembership{Member, IEEE}, \\ and Sonia Mart{\'\i}nez, \IEEEmembership{Fellow, IEEE}
\thanks{This work is supported by the ARL grant: W911NF-23-2-0009.}
\thanks{Nirabhra Mandal, Mohammad Khajenejad
and Sonia Mart{\'\i}nez are with the Mechanical \& Aerospace Engineering Department, University of California San Diego. \texttt{\{nmandal,mkhajenejad,soniamd\}@ ucsd.edu}.}}

%\onecolumn

\maketitle

\begin{abstract}
  This paper provides a novel solution to a task allocation problem,
  by which a group of agents decides on the assignment of a discrete
  set of tasks in a %an optimal and
  distributed manner. In this setting, heterogeneous agents have
  individual preferences and associated rewards for doing each task;
  however, these rewards are only known asymptotically. We start by
  formulating the assignment problem by means of a combinatorial
  partition game for known rewards, with no constraints on number of
  tasks per agent. We relax this into a weight game, which together
  with the former, are shown to contain the optimal task allocation in
  the corresponding set of Nash Equilibria (NE). We then propose a
  projected, best-response, ascending gradient dynamics (\dynacr) that
  converges to a NE in finite time. This forms the basis of a
  distributed online version that can deal with a converging sequence
  of rewards by means of an agreement sub-routine. We present
  simulations that support our results.
\end{abstract}


%%
%\margin{refer to EF-FB:03r for task allocation}
%\marginn{done}
%%
\section{Introduction}\label{sec:intro}
A prototypical multi-agent coordination problem aims to find an
efficient assignment of group of agents to complete a collection of
tasks.  These tasks can range from abstract sets of objectives to
specific physical jobs, the nature of which may not be completely
known. In addition, the agents composing the group may have
heterogeneous capabilities, and react to different sets of incentives
that are being learned progressively. This necessitates of novel
task-assignment algorithms that can adapt and react online as new
information arises. Motivated by this, we study a discrete task
allocation problem modeled as a game of self-interested agents that
have partial knowledge of their rewards. This requires addressing the
problem's combinatorial nature, and designing provable-correct
distributed dynamics that adapt to dynamic rewards revealed online. To
the best of our knowledge, algorithms that combine all these features
are not available in literature.
%%
% \margin{Final motivation sentence be revised: we should include a
% sentence saying something like: algorithms that account for all of
% these constraints are not availalbe in the literature....}
% \marginn{done}
%%

\paragraph{Literature review}
The problem of task allocation with known rewards has been widely
considered; see e.g.~\cite{HLC-LB-JPH:09, FB-EF-MP-KS-SLS:11,
  AS-SLS:17}. A centralized solution to this problem, where the number
$m$ of tasks and agents are equal and a task-agent matching is sought,
is the optimization-based Hungarian algorithm~\cite{HWK:55}, and its
distributed version~\cite{SC-GN-MR-ME:17}. The latter, which
reproduces the Hungarian algorithm locally, requires tracking of the
agents' identities associated with each task, has a time complexity of
$O(m^3)$ and communication cost of $O(m\log m)$ (per communication
round). Thus, the algorithm can be computationally and
memory-intensive for large problems, and hard to adapt as new tasks are generated or
their valuations change online.
% However, the Hungarian method has a time complexity of $O(n^3)$ and
% a space complexity of $O(n^2)$, with $n$ being the number of agents.
%
% \margin{What hungarian method, the centralized one? It seems to me
%   this is the complexity of the centrealized one because this only
%   refers to total convergence time and not communication complexity. I
%   should say once again that we are only interested in the comparsion
%   of the decentralized hungarian algorithm, not the centralized
%   one. The communication complexity should be larger. Plus, it is also
%   important to talk about memeory complexity.}  \marginn{updated}
% 
% This makes it less suitable for very large problems, as it can be
%computationally expensive and memory-intensive.
% Moreover, the task assignment formulation is based on costs, not
% rewards, all elements in the cost matrix are non-negative, otherwise
% they need to be converted to non-negative values, which can
% potentially alter the problem's meaning.
%
%%
%\margin{Is that really a big problem? I'll omit that sentence for now}
%\marginn{noted}
%%
%
The work in~\cite{JC-AF-PM-SDS:14} provides a tractable, sub-optimal
solution to the same NP-hard problem, 
%%
%\margin{what are the differences in formulation between the problems in these papers and the problems the Hungarian algorithm resolves?}
%\marginn{discussed}
%%
while the research in
\cite{AP-SS-HLC:18} showed that the sub-optimality can be resolved by
restricting heterogeneous agents to be of certain types. 
%%
%\margin{is this a distributed solution?}
%\marginn{no. There is a central algorithm that decides which type of agent to put in what task}
%%
In the same vein, the works
in~\cite{NR-SSK:23,JV:08, GC-CC-MP-JV:11} considers submodular
functions which allow rewards to take any non-negative value. However,
submodular optimization can be applied in specific domains where the
property naturally arises, such as in certain economics and
distributed sensing problems.  Alternatively, a well known approach to
(unconstrained) task assignment problems is given by $k$-means
clustering and the Lloyd's algorithm~\cite{SL:82}. 
%%
%\margin{we have reference to Lloyd's algorithm for sure in our bib file.}  
%\marginn{updated}
%%
By interpreting that tasks are generated by a probability distribution,
the approach can handle tasks generated
dynamically~\cite{MS-YDM-ME:18,MS-ME:18,JC-SM-TK-FB:02-tra,EF-FB:03r}. 
%%
%\margin{cite also Francesco's and emilio frazzoli's work on task assignment}
%\marginn{done}
%%
% and apply Lloyd's algorithm~\cite{lloyd1982least}, where if the cost
% function follows a monotonic structure, \emph{Voronoi} partitions
% can be used to compute optimal partitions of the underlying space
However, it is well known that Lloyd's algorithm is sensitive to the
initial task assignment for a small number of agents, and converges to
a local minima.
% so depending on the initial configuration, it may converge
% to different local minima, potentially resulting in suboptimal
% assignments.

Game-theoretic models have also been proposed to find solutions to
task allocation problems. For sensor networks, each agent is equipped
with an appropriate utility function~\cite{MZ-SM:10,JRM-GA-JSS:09,RK-RC-DG-JRM:22}
%%
%\margin{we do cite Jason Marden and Jeff shamma, don't we?, pls include in here} 
%\marginn{done}
%%
and the optimal task allocation is related to the Nash equilibrium of this
game. Any Nash-seeking~\cite{PF-MK-TB:10} algorithm returns a
solution; but often, these algorithms require strong assumptions on
the utility functions and their derivatives. Distributed versions of
Nash-seeking with consensus in continuous time have been explored
in~\cite{MY-GH:17,ZD-XN:18}, while~\cite{JK-AN-UVS:12,FS-LP:16}
addressed the problem in discrete time. All the work in
\cite{MY-GH:17}-\cite{FS-LP:16} assumed complete and perfect
information for agents, while in practice, agents may have limited or
imperfect information about the tasks, and the capabilities of other
agents. Potential games can be used in this regards, but they do not
work when the reward parameters are unknown. To that end,
\cite{RK-RC-DG-JRM:22} characterizes transient behavior for set
covering games, and~\cite{ACC-RAM-RK-NRJ:10} looks at a general
potential game approach for task allocation. In the latter case, the
agents are homogeneous and tasks have same rewards for all agents.

\paragraph{Contributions}
We consider a task-assignment problem where a number of agents is to
be matched to an unrestricted set of tasks. In the considered
formulation, the number of tasks per agent is not constrained, yet the
optimal assignment problem remains combinatorial as the number of
tasks is discrete. To deal with arbitrary heteregenous agents, we
derive a game-theoretic partition problem formulation that favors task
distribution. We then relax the game into a weight game, one per
task. We obtain characterizations of the NE of each game, their
relationship, and identify conditions under which the NE leads to an
optimal solution of the original assignment problem. Leveraging the
relaxed formulation, and under a full-information assumption, we
derive a projected best-response dynamics that is shown to converge to
the optimal task allocation in finite time. This forms the basis for a new algorithm,
\dynacr, which is distributed, does not require the knowledge of other
agent identities, and converges to the optimal task allocation, also in finite time, as
rewards are revealed online.

\section{Preliminaries} \label{sec:prelim}
Here, we formalize the notations and briefly list some  well-known concepts that are used to solve the problem formulated in the following section. 

\subsection{Notations}
The sets of real numbers,  non-negative real
numbers, and non-negative integers are denoted
as $\real$, $\real_{\geq 0}$, and $\intpos$,
respectively. For a set $\set$, $|\set|$ denotes its cardinality, $2^\set$ represents the class of all its subsets, $\set^n$ denotes the $n$ Cartesian product of $\set$ with itself, and $\set^{n \times m}$ collects all
$n \times m$ matrices whose $(i,j)\tth$ entry lies in $\set$.  Given $\M \in \set^{n \times m}$, $m^j_i$ is its $(i,j)\tth$ entry, and $\bld{m}_i^\top \in \set^m$ (resp.~$\bld{m}^j \in \set^n$) its $i\tth$ row (resp.~its $j\tth$ column). %Moreover, if $f$ is a function, then $f(\set) \ldef \{f(x) \,|\, x \in \set\}$.
For $x \in \real$, $[x]_0^1 \ldef \max \{0, \min \{ x , 1 \} \}$. For a set $\set$, define $\submax \set \ldef \max \{s \in \set \,|\, s \neq \max \,\set\}$. For a vector $\x \in \real^n$ and a set $\set \subseteq \real^n$, $d(\x,\set) \ldef \inf_{\y \in \set} \|\x-\y\|_1$ is the distance of the vector from the set. Lastly, the empty set is denoted as $\varnothing$. 
%%
%\margin{I've removed the notation for the image of the mapping. It should be fine}
%\marginn{noted}
%%

\subsection{Game theory}
A \emph{strategic form game}~\cite{YN:14} is a tuple $\game \ldef \left<\agt,\{\S_i\}_{i\in \agt}, \{\psi_i\}_{i \in \agt} \right>$ consisting of the following components:
\begin{enumerate}
	\item a set of \emph{players} (or \emph{agents}) $\agt$;
	\item a set of \emph{strategies} $s_i \in \S_i$ available to each $i \in \agt$;
	\item a set of \emph{utility functions} $\psi_i : \times_{i \in \agt}\S_i \to \real$ over the strategy profiles of all the agents.
\end{enumerate}
In what follows,  $s_{-i}$ denotes the strategy profile of all players other than $i \in \agt$. Next, we formally state the definition of the NE of a strategic form game.

\begin{define}[Nash equilibrium]
\label{def:ne}
The strategy profile $(\hat{s}_i,\hat{s}_{-i})$ is a Nash equilibrium
(NE) of $\game$ if and only if
\begin{align*}
	\psi_i(\hat{s}_i,\hat{s}_{-i}) \geq \psi_i(s_i,\hat{s}_{-i}), \,\, \forall s_i \in \S_i, \,\, \forall i \in \agt\,. 
\end{align*}
$\ne(\game)$ denotes the set of all Nash equilibria of $\game$.
\bulletend
\end{define}

\subsection{Graph theory}
A \emph{directed graph}~\cite{RD:17}
%%
%\margin{put all the graph theory  concepts in textit form (as as those with games)}
%\marginn{done}
%%
$\grph \ldef (\agt,\edg)$, is a tuple consisting of (a)  a set of \emph{nodes} (here agents $\agt$); (b) a set of \emph{arcs} $\edg \subseteq \agt \times \agt$ between the nodes.
%% 
%\margin{I have removed the enumerate environment for this def, it takes space and does not seem necessary}
%\marginn{noted}
%% 
The set $\neigh_i \ldef \{ j \in \agt \,|\, (j,i) \in \edg\}$ denotes the (in) \emph{neighbors} of node $i \in \agt$ and $\neighb_i \ldef \neigh_i \cup \{i\}$. A \emph{path} is an ordered set of non-repeating nodes such that each tuple of adjacent nodes belongs to $\edg$. The graph $\grph$ is said to be strongly connected if there exists a path from every node to every other node. The \emph{diameter} of the graph $\diam(\grph)$ is the length of the largest possible path between any two nodes.
%%
%\margin{If there is room, please refer to a book on graph theory or the motion coordination book}
%\marginn{done}
%%

\section{Problem Formulation} \label{sec:problem} 
A group of agents $\agt \ldef \until{n}$ is to complete a set of tasks
$\tsk \ldef \until{m}$, where $n \neq m$ possibly, 
in a distributed manner.  For this purpose,
each agent $i \in \agt$ encodes via
$\phi_i : \tsk \to \realnonnegative$ the importance of each task (the
higher $\phi_i(q)$ the larger the agent's capability/fondness on
$q \in \tsk$) and $\rf_i : \tsk \to \realnonnegative$ the 
reward for completing each task.  
For the sake of brevity, define
$f_i(q) \ldef \rf_i(q)\phi_i(q) \geq 0$, $\forall q \in \tsk$,
$\forall i \in \agt$. 
An optimal task assignment is the solution to 
%% 
%\margin{I've rephrased this part, hopefully it won't raise questions}
%\marginn{}
%%
%maximize the utility
%\begin{align}
%  J(\salloc) \ldef \sum_{i \in \agt} \sum_{q \in \alloc_i} \rf_i(q)\phi_i(q).
%  \label{eq:partition_reward}
%\end{align}
\begin{subequations}
\begin{align}
	\label{eq:partition_reward} & \max_{{\salloc = (\alloc_1,\cdots,\alloc_n)  \subseteq \tsk^n}} \quad J(\salloc)\ldef \sum_{i \in \agt} \sum_{q \in \alloc_i} f_i(q),   \\
	\label{eq:partition_constraint} & \mathrm{s.t.} \quad \bigcup_{i \in \agt} \alloc_i = \tsk; \quad \alloc_i \cap \alloc_j = \varnothing,\,\, \mathrm{if}\,\, i \neq j.
\end{align}
\label{eq:partition_optimization}
\end{subequations}
Here, $\alloc_i \subseteq \tsk$ is the set of tasks assigned to agent
$i \in \agt$ and $\salloc = (\alloc_1,\cdots,\alloc_n)$ is the ordered
collection of sets that defines a partition of $\tsk$ (as in~\labelcref{eq:partition_constraint}). 
%Thus,
%$\cup_{i \in \agt} \alloc_i = \tsk$ and
%$\alloc_i \cap \alloc_j = \varnothing$, $\forall i,j \in \agt$ such
%that $i \neq j$.

% This problem can be solved in a centralized manner by assigning a task
% $q \in \tsk$ to the agent $i \in \agt$ for which $f_i(q)$ is the
% largest, while ties can be broken arbitrarily. This motivates the
% following definition.

The group of agents is to compute an optimal partition of the task set
$\tsk$ on their own. Naturally, each agent $i \in \agt$ aims to get
the tasks $q\in \tsk$ for which $f_i(q)$ is the largest. This
motivates the following definition. 
%%
%\margin{I've changed this so that we don't mention the centrealized solution}
%\marginn{noted}
%%
\begin{define}[Task specific dominating agent] \label{def:dominating_agent} 
%\thmtitle{Dominating agent
  %  for a task} 
    An agent $i \in \agt$ is said to be a \emph{dominating
    agent for task $q \in \tsk$} (or $i$ \emph{is dominating for} $q$), if $f_i(q) \geq f_j(q)$,
  $\forall j \in \agt$.  If a task $q \in \tsk$ has exactly one
  dominating agent, we say that there exists a \emph{unique dominating
    agent for task $q$}.  \bulletend
\end{define}
 The collection of possible strategies for each agent is a combinatorial class, which
grows exponentially as the number of tasks $m$ increases. To address this problem, 
we first assume that each $i \in \agt$ 
measures the utility of a subset of tasks $
\alloc_i \subseteq \tsk$ via

\vspace{-.3cm}
\begin{align}
	H_i(\alloc_i,\alloc_{-i}) \ldef \sum_{q \in \alloc_i} \Big[ f_i(q) - \hspace{-2ex} \max_{j \in \agt \setminus \{i\}, q \in \alloc_j} f_j(q) \Big].
	\label{eq:agent_reward_partn} 
\end{align}
%%
%\margin{I'm not sure what is the meaning of the expression  ${j \in \agt \setminus \{i\}, \alloc_i \cap \alloc_j \neq \varnothing}$. It should just be the $j \neq i$? with no reference to these set intersections?}
%\marginn{updated as per discussion}
%%
%This has the interpretation that agent $i$ does not want to dispute tasks with other agents.
This leads to the partition
game
%%
%\margin{Observation:  introduce formally the notation for a game (elements n the tuple), as well as the notion of deterministic NE for the game before the next section}
%\marginn{I'm making a section of preliminaries. We can decide whether to put it in introduction or not}
%%
\begin{align*}
	\game_P \ldef {\left<\agt,\{2^{\tsk}\}_{i\in \agt}, \{H_i\}_{i \in \agt} \right>},
\end{align*}
where the strategy of each agent is to choose a subset $\alloc_i$ of
$\tsk$ to maximize $H_i$. In this way, 
agent strategies are no-longer required to form a valid partition,
but the utility in~\labelcref{eq:agent_reward_partn} penalizes each agent for
taking tasks that others have chosen.
%%
%\margin{We probably
%  should cite the siam work coauthored by Minghui Zhu on coverage control games --if I remember correctly this type of utilities are related to those
%  defined by Minghui. Please check it out.}
%\marginn{noted}
%%

Second, we further relax this game by reducing the decision of each
agent $i \in \agt$ regarding task $q \in \tsk$ to the computation of a
weight $w^q_i \in [0,1]$.  Briefly, this defines
$\W \in [0,1]^{n \times m}$ as the matrix whose $(i,q)\tth$ entry is
$w^q_i$.
%%%%%%%%%%%%%%%%%%%%%%%%%%%%%% 
Thus, 
%the $i\tth$ row (resp.~the $q\tth$ column) of $\W$ is denoted by 
$\w_i^\top \in [0,1]^m$ (resp.~$\w^q \in [0,1]^n$) %and 
represents the weights that agent $i \in \agt$ (resp.~for task
$q \in \tsk$) gives to each task (resp.~given by each agent).
%%%%%%%%%%%%%%%%%%%%%%%%%%%%%%% 
%
%\marginn{check here}
%
% , while $\w^q$ is
% the $q\tth$ column of $\W$ and represents the weights of all agents
% for task $q \in \tsk$.
Agent $i \in \agt$ is equipped with the
utility function:
\begin{align}
	U_i(\w_i,\w_{-i}) & = \sum_{q \in \tsk} \Big[ f_i(q)w^q_i - \max_{j \in \agt \setminus \{i\}} f_j(q)w^q_j w^q_i \Big] \,,
	\label{eq:agent_reward_max} 
\end{align}
which collectively define the weight game
\begin{align*}
	\game_W \ldef {\left<\agt,% \{[0,1]\}_{i\in \agt},
	\W,\{U_i\}_{i \in \agt} \right>},
\end{align*}
%with the weights encoding strategies.
In this way, a product of
weights in the second part of the sum in~\labelcref{eq:agent_reward_max}
relaxes the check on overlapping task in~\labelcref{eq:agent_reward_partn}. 
%%
%\margin{OK, we still have intersections, but for each task. Feel free to rephrase this in the light of the new definition}
%\marginn{updated, please check}
%%
In this paper, we ignore the trivial
case where all agents get the same payoff for a task, as stated in the
following.

\begin{assume}[Non-trivial task assignment]\label{asmp:not_all_dominating}
%\thmtitle{Non-trivial task assignment} 
  Not all agents are dominating for each task $q \in \tsk$.
  \bulletend
\end{assume}

The above framework allows us to deal with a case where the $f_i(q)$
are unknown to the agent, but where these values can be learned
progressively by an external mechanism until convergence. More
precisely, we assume the following. 
%%
%\margin{It doesn't matter if ri or phii are known or not, the agent is learning fi}
%\marginn{correct, noted}
%% 
   % However, we assume that each agent can
% construct better and better estimates that converge to $\rf_i(q)$ for
% each task $q \in \tsk$. Since the $\phi_i$ values are constant, we
% make the following assumption.
%\begin{assume}\label{asmp:convering_seq_reward}
%\thmtitle{Converging reward sequence}
%There exists a sequence $\{\fseq^q_i(t)\}_{t \in \intpos}$ with the property,
%%\begin{align*}
%%	\fseq^q_i(t) \to f_i(q) \,\, \mathrm{as} \,\, t \to \infty, \,\, \forall q \in \tsk, \forall i \in \agt\,.
%%\end{align*}
%\begin{align*}
%	\begin{array}{cr}
%		\fseq^q_i(t) \to f_i(q) \,\, \mathrm{as} \,\, t \to \infty, \,\, \forall q \in \tsk, \forall i \in \agt\,. & \qquad\qquad\,\,\,\,\bullet
%	\end{array}
%\end{align*}
%\bulletend
%\end{assume} 
\begin{assume}[Converging reward sequence]\label{asmp:convering_seq_reward}
%\thmtitle{Converging reward sequence}
For each $i \in \agt$, $q \in \tsk$, there exists a sequence $\{\fseq^q_i(t)\}_{t \in \intpos}$ such that $\fseq^q_i(t) \to f_i(q)$ as $t \to \infty$.
\bulletend
\end{assume} 
In what follows, we first study the games when the reward parameters
are known. Then we adapt the results for the case when only a
converging reward sequence is available.

Now we formally state the goals of this work.
\begin{problem}\label{prob:main}
  Given the aforementioned setup and %Assumption~\ref{asmp:not_all_dominating},
  the non-trivial task assignment assumption,  find
\begin{enumerate}
	\item \label{prob:relate_nep_neg} a relationship between the NE of $\game_P$ and $\game_W$,
	\item \label{prob:relate_ne_opt} a relationship between the NE and optimal partitions according to~\labelcref{eq:partition_optimization},
	\item \label{prob:dist_ne_algo} a  distributed algorithm that converges to the NE of the limiting weight game $\game_W$ under the converging reward sequence assumption. 
\bulletend
\end{enumerate}
\end{problem}

\section{On Nash Equilibria and Optimal Partitions} \label{sec:ne}

 We start by addressing the first two problems above. %Problems~\ref{prob:main}~(\ref{prob:relate_nep_neg}) and~\ref{prob:main}~(\ref{prob:relate_ne_opt}).
%%
%\margin{use appropriate labels for referring to each part of the problem. }
%\marginn{done}
%%
%the Nash equilibria of the weight game
%\begin{align}
%	& \notag \ne(\game_W) \ldef \{(\hat{\w}^i,\hat{\w}^{-i}) \in [0,1]^n \,\,|\,\,\\
%	& \quad U_i(\hat{\w}^i,\hat{\w}^{-i}) \geq U_i(\w_i,\hat{\w}^{-i}), \forall\, \w_i \in [0,1], \forall i \in \agt\},
%	\label{eq:ne_weight_game}
%\end{align} 
%and the optimal partition $\salloc^* = (\alloc^*_1,\cdots,\alloc^*_n)$ that solves the optimization problem 
%\begin{align}
%	\notag \max_{\substack{\salloc = (\alloc_1,\cdots,\alloc_n) \\ \subseteq \tsk^n}} J(\salloc) \quad \mathrm{s.t.} & \bigcup_{i \in \agt} \alloc_i = \tsk; \\ & \alloc_i \cap \alloc_j = \varnothing\,\, \mathrm{if}\,\, i \neq j;
%	\label{eq:partition_optimization}
%\end{align}
%with $J$ as in~\labelcref{eq:partition_reward}.
Thus, 
we first characterize
the NE of the partition game~$\game_P$.
\begin{lemma}[Nash equilibria of $\game_P$] \label{lem:nash_partition_game}
%\thmtitle{Nash equilibria of $\game_P$}
The strategy $(\alloch_i,\alloch_{-i}) \in \ne(\game_P)$ if and only if:% $(\alloch_i,\alloch_{-i})$ has the following properties: 
\begin{enumerate}
\item \label{prop:part_game_dom_agt} for each  $q\in \tsk$, $\exists \, i \in \agt$  dominating for $q$ and $q \in \alloch_i$;
	\item \label{prop:part_game_not_dom_agt} if $j$ is not a dominating agent for task $q$, then $q \not\in \alloch_j $.
	%\item $\bigcup_{i \in \agt} \alloch_i = \tsk$, $\forall i,j \in \agt$ such that $i \neq j$, $\supp(\wb^i) \cap \supp(\wb^j) = \nullset$;
\end{enumerate}
\end{lemma}

\begin{proof}
  First, we show the necessity of Properties~\ref{prop:part_game_dom_agt} and~\ref{prop:part_game_not_dom_agt}.
Suppose that $(\alloch_i,\alloch_{-i}) \in \ne(\game_P)$.
  We prove Property~\ref{prop:part_game_dom_agt}
by contradiction and assume $\exists \,q \in \tsk$ such that $\forall i \in \agt$ dominating for task $q$, $q \notin \alloch_i$. Pick one such agent $i$ and take $\alloc_i = \alloch_i \cup \{q\}$. Then, 
%%
%\margin{to be clear, we can just say $H_i(\alloc_i,\alloch_{-i}) - H_i(\alloch_i,\alloch_{-i})  = f_i(q) - \max_{{k \in \agt \setminus \{i\}, \, \alloc_i \cap \alloch_k \neq \varnothing}} f_k(q)$, then avoid the Delta f notation} 
%\marginn{updated}
%%
%%
%\margin{fix the notation, as with the H } 
%\marginn{fixed}
%%
%$\gamma f_i^q \ldef f_i(q) - \max_{k \in \agt \setminus \{i\}, \, q \in \alloch_k} f_k(q)$, 
%yields
\begin{align*}
	H_i(\alloc_i,\alloch_{-i}) - H_i(\alloch_i,\alloch_{-i}) = f_i(q) - \hspace{-1ex} \max_{k \neq i, \, q \in \alloch_k} f_k(q)>0 .
\end{align*}
The inequality is strict since $i$ is dominating  and the max is over all agents that are not dominating for $q$ (by assumption). This is a contradiction with  $(\alloch_i,\alloch_{-i}) \in \ne(\game_P)$.

The necessity of Property~\ref{prop:part_game_not_dom_agt} also follows from contradiction. Suppose $\exists \, q \in \tsk$ and a $j \in \agt$ 
%%
%\margin{$j$ is an agent,  it can't be in $\tsk$. Pls fix the range}
%\marginn{fixed}
%% 
not dominating for $q$ with $q \in \alloch_j$. From Property~\ref{prop:part_game_dom_agt}, there is an
$i \in \agt$ dominating for $q$ with $q \in \alloch_i$. Thus, with strategy $(\alloch_j,\alloch_{-j})$, 
%%
%\margin{This $\gamma f_i^q$ is not the same as the $\gamma f_i^q$ above (and we are in the context of the same proof) }
%\marginn{updated}
%% 
\begin{align*}
f_j(q) - \max_{k \neq j, \, q \in \alloch_k} f_k(q) = f_j(q) - \max_{k \in \agt} f_k(q) < 0. 
%%	
%\margin{If we are referring to $j$, we should refer to $H_j$ and not $H_i$ for some other $i$, right?}
%\marginn{fixed}
%%
\end{align*}
Now, consider the strategy $\alloc_j = \alloch_j \setminus \{q\}$. It
follows that
$H_j(\alloc_j,\alloch_{-j}) - H_j(\alloch_j,\alloch_{-j}) > 0$, which contradicts 
$(\alloch_j,\alloch_{-j}) \in \ne(\game_P)$.

Now, we show the sufficiency of Properties~\ref{prop:part_game_dom_agt} and~\ref{prop:part_game_not_dom_agt}. Let $(\alloch_i,\alloch_{-i})$ satisfy Properties~\ref{prop:part_game_dom_agt} and~\ref{prop:part_game_not_dom_agt} and let $i \in \agt$ be an arbitrary but fixed agent. Suppose that $\alloc_i 
\neq \alloch_i$ is any other strategy. Then, 
 the proof follows from  three cases:
%%
%\margin{Let's not use the itermize environment, it saves space}
%\marginn{noted}
%%
%\begin{itemize}
%\item[]

  \emph{Case (i):} $\exists \, q \in \alloc_i$ 
%%  
%\margin{sometimes you write $\exists \,$, and sometimes you omit the space. The space looks better, so please add it everywhere} 
%\marginn{noted and updated}
%%
such that $q \notin \alloch_i$ and $i$ is  dominating for $q$. Then, since $\exists \, j \in \agt$ 
%%  
%\margin{pls add the set where all agents belong to, everywhere}
%\marginn{updated} 
%%
dominating for $q$ with $q \in \alloch_j$, %we obtain
\begin{align*}
  f_i(q) - \max_{k \neq i, \, q \in \alloch_k} f_k(q) = f_i(q) - f_j(q) = 0. 
\end{align*}
%	The last equality comes from the definition of a dominating agent.

\emph{Case (ii):} $\exists \,q \in \alloc_i$ such that
$q \notin \alloch_i$ and $i$ does not dominate $q$. Then, as
$\exists \, j \in \agt$ dominating for $q$ and $q \in \alloch_j$, we
have
\begin{align*}
	f_i(q) - \max_{k \neq i, \, q \in \alloch_k} f_k(q) = f_i(q) - f_j(q) < 0.
\end{align*}

\emph{Case (iii):} $\exists \,q \in \alloch_i$ such that
$q \notin \alloc_i$. This can only happen if $i$ is dominating for $q$
(else, by Property~\ref{prop:part_game_not_dom_agt}, $q \notin \alloch_j$).
Then, %Then, note that
\begin{align*}
  f_i(q) - \max_{k \neq i, \, q \in \alloch_k} f_k(q) \geq 0.
\end{align*}

From the above, it is easy to see that any deviation from $(\alloch_i,\alloch_{-i})$ 
%%
%\margin{relate the deviations to these differences for more clarity}
%\marginn{updated}
%%
will not result in an increase in utility for $i$ since $H_i(\alloc_i,\alloch_{-i}) - H_i(\alloch_i,\alloch_{-i}) = f_i(q) - \max\limits_{k \neq i, \, q \in \alloch_k} f_k(q)$.
%%
%\margin{read up to here}
%%
\end{proof}

From the previous result, at least one of the dominating agents will
be assigned to a task by means of a NE strategy of $\game_P$. However,
this does not preclude that two dominating agents are assigned the
same task. Next, we show that the NE of the relaxed game $\game_W$ are
equivalent to the NE of~$\game_P$.

\begin{lemma}[Nash equilibria of $\game_W$] \label{lem:nash_weight_game}
%\thmtitle{Nash equilibria of $\game_W$}
The strategy $(\wh_i,\wh_{-i}) \in \ne(\game_W)$ if and only if:
%%
%\margin{Again, use labels for these items, so we don't have to refer to them ``by hand''}
%\marginn{updated}
%%
\begin{enumerate}
%	\item \label{prop:wt_game_feasible} $\hat{w}^q_i \in [0,1]$, $\forall i \in \agt$, $\forall q \in \tsk$;
	\item \label{prop:wt_game_dom_agent} for each  $q\in \tsk$, $\exists \, i \in \agt$  dominating for $q$ and $\hat{w}^q_i = 1$;
	\item \label{prop:wt_game_not_dom_agent} if $j$ is not a dominating agent for task $q$, then $\hat{w}^q_j = 0 $.
	% \item $\bigcup_{i \in \agt} \supp(\hat{\w}^i) = \tsk$; $\forall i,j \in \agt$ such that $i \neq j$, $\supp(\wb^i) \cap \supp(\wb^j) = \nullset$;
\end{enumerate}
%\begin{align}
%	\ne(\game_P) = \setfunc(\ne(\game_W))\,.
%\end{align}
%%
%\margin{Please state the range and domain of the mapping $\setfunc$, so that we can see they match in the previous equality. I think the definition of $\setfunc(\W)$ should be a product}
%\marginn{Added before \ref{eq:support_func}}
%%
\end{lemma}
\begin{proof}
First, we show the necessity of all properties. Suppose $(\hat{\w}_i,\hat{\w}_{-i}) \in \ne(\game_W)$. %Then Property~\ref{prop:wt_game_feasible} holds from the Definition~\ref{def:ne}. 
%%
%\margin{fix this label if you want to use it, I have remove some above. In any case, isn't this property trivial? If so, I would remove it from the list of properties.}
%\marginn{updated}
%%
We show Property~\ref{prop:wt_game_dom_agent} is necessary by contradiction. Consider an arbitrary $q \in \tsk$ and suppose that for all dominating agents $i^*_q \in \agt$  for task $q$, it holds that $\what^q_{i^*_q} < 1$. In particular, for any such $i^*_q$, we have $\max_{j \neq \{i^*_q\}}f_j(q)\what^q_j < f_{i^*_q}(q)$. %This comes from the assumption that $\what^q_i < 1$, $\forall i \in \agt$ such that $i$ is a dominating agent for task $q$ (and the fact that $f_j(q)w^q_j < f_{i^*_q}(q)$, $\forall w^q_j \in [0,1]$ if $j$ is not a dominating agent).
%%
%\margin{I've commented this out because it is clear I think.}
%\marginn{noted}
%%
Now consider the strategy $\w_{i^*_q}$, where $w^q_{i^*_q} = 1$
%%
%\margin{Please revise the order of all of the indices in all the proof. Recall that the indices for agents go below, for tasks go upstairs :) this includes the $\w$! Most indices need to be changed}
%\marginn{yes, sorry I missed a few}
%%
and $w^{i^*_q}_p = \what^{i^*_q}_p$, $\forall p \neq q \in \tsk$. Then,
\begin{align*}
	& U_{i^*_q}(\w_{i^*_q},\wh_{-i^*_q}) - U_{i^*_q}(\wh_{i^*_q},\wh_{-i^*_q}) = \\
	& \hspace{3em} \Big[ f_{i^*_q}(q) - \max_{j \neq i^*_q} f_j(q)\what^q_j \Big] \big[ 1 -  \what^q_{i^*_q} \big] > 0.
\end{align*} 
This leads to a contradiction with $(\hat{\w}_i,\hat{\w}_{-i}) \in \ne(\game_W)$.

We similarly show Property~\ref{prop:wt_game_not_dom_agent} is necessary by contradiction. Let $q \in \tsk$ be an arbitrary task, and suppose that $\exists \,j \in \agt$ which is not dominating for $q$ but for which $\what^q_j > 0$. Due to Property~\ref{prop:wt_game_dom_agent}, let $i^*_q$ be the dominating agent for $q$ such that $\what^q_{i^*_q} = 1$. Now define a new strategy $\w_j$, with $w^q_j = 0$ and $w^j_p = \what^j_p$, $\forall p \neq q \in \tsk$. Then,
\begin{align*}
	U_j(\w_j,\wh_{-j}) - U_j(\wh_{j},\wh_{-j}) = 
	 \big[f_j(q) - f_{i^*_q}(q) \big] [- \what^q_j ] > 0, 
\end{align*} 
where the inequality is because both terms are negative. This
contradicts $(\hat{\w}_i,\hat{\w}_{-i}) \in
\ne(\game_W)$. %, proving the necessity of Property~\ref{prop:wt_game_not_dom_agent}.

Next, we show sufficiency. Let
$(\hat{\w}_i,\hat{\w}_{-i}) \in [0,1]^{n \times m}$
%%
%\margin{say where it lives}
%\marginn{done}
%%
be a candidate strategy satisfying
Properties~\ref{prop:wt_game_dom_agent}-~\ref{prop:wt_game_not_dom_agent}
and let $i \in \agt$. Take any other $\w_i \neq \wh_i$ and a task
$q \in \tsk$. The proof follows from the following cases. 
%\mo{by defining $\tilde{\gamma}f^q_i \ldef f_i(q) - \max_{j \neq i} f_j(q)\what^q_j,\hat{\gamma}f^q_i \ldef f_i(q) - \max_{k \in \agt \setminus \{i\}} f_k(q)\what^q_k$, we have}
%%
%\margin{First: aren't these two differences the same? Second: we don't need to use a tilde or hat notation for these. I'd just denote them as $\gamma f_i^q$. Since the previous time we used this notation was for the proof of the previous result, (and the notation is just meant as a shorthand), then we can totally repeat it.  Third: actually, can we just remove these notations?  unless they are very helpful, I prefer to avoid introducing new notation. And while in the first proof, I understood the Delta f was related to the difference in H,   and because of that it could be helpful (...); However, they don't seem  that helpful here.  }
%\marginn{removed the delta notation} 
%%
%%
%\margin{rephrase this type of sentences, ``study all cases by studying the following cases''. You tend to do a lot of word repetition, and you tend to write too wordy, please be more concise in the rest of the paper.}
%\marginn{updated}
%%
%\begin{itemize}
%\item[]

  \emph{Case (i):} $i$ is a dominating agent for task $q$ and $\what^q_i = 1$. Then, from Definition~\ref{def:dominating_agent},
%%
%\marginn{this equation is modified to save space}
%%
%	\begin{align*}
%	& \Big[ f_i(q) - \max_{j \in \agt \setminus \{i\}} f_j(q)\what^q_j \Big] w^q_i \\
%	& \qquad \qquad \leq \Big[ f_i(q) - \max_{j \in \agt \setminus \{i\}} f_j(q)\what^q_j \Big] \what^q_i \,.
%	\end{align*}
\begin{align*}
	\Big[f_i(q) - \max_{j \neq i} f_j(q)\what^q_j \Big]w^q_i \leq \Big[ f_i(q) - \max_{j \neq i} f_j(q)\what^q_j \Big] \what^q_i.
\end{align*}
%	The inequality comes from the definition of a dominating agent and the fact that $\what^q_j \in [0,1]$.
	
      %\item[]

        \emph{Case (ii):} $i$ is a dominating agent for task $q$ and ${\what^q_i < 1}$. Then, since $\exists \, j \in \agt$ dominating for $q$ and $\what^q_j = 1$,
\begin{align*}
	& \Big[f_i(q) - \max_{k \neq i} f_k(q)\what^q_k \Big]w^q_i = \big[ f_i(q) - f_j(q) \big] w^q_i \\ 
	& =	\big[ f_i(q) - f_j(q) \big] \what^q_i = \Big[f_i(q) - \max_{k \neq i} f_k(q)\what^q_k \Big] \what^q_i = 0,
\end{align*}
	%The second and third equality is because 
	since $f_i(q) - f_j(q) = 0$.
	
%%	
%\margin{OK, pls explain the second equality in magenta. Why did $w_i^q$ changed to $\hat{w}_i^q$? Maybe that is another line and we shouldn't include the equal sign? I have a similar question about case iii) below. Why do we replace similarly the weights there?}
%\marginn{I have added an explanation for case (ii) here. I had that before, but I removed it to make the arguments brief. I guess it became too brief. I have added an explanation for case (iii) also. Please see if we should keep these}
%\margin{yes, this is enough.}
%%
      %\item[]
        \emph{Case (iii):} $i$ is not a dominating agent for task $q$ (and hence $\what^q_i = 0$).
	Again, $\exists \, j \in \agt$ dominating for $q$ and $\what^q_j = 1$. Then,
\begin{align*}
	 & \Big[f_i(q) - \max_{k \neq i} f_k(q)\what^q_k \Big]w^q_i = \big[ f_i(q) - f_j(q) \big] w^q_i \\
	 & < \big[ f_i(q) - f_j(q) \big] \what^q_i = \Big[f_i(q) - \max_{k \neq i} f_k(q)\what^q_k \Big] \what^q_i,
\end{align*}
	since $f_i(q) < f_j(q)$.
        % \end{itemize}
%%        
%\margin{let's avoid the word ``again''. The reason is that we don't want to give the reviewers the idea that we are repeating the proof. }
%\marginn{noted}
%% 
Now, using these three cases, it is
        easy to see that any deviation from
        $(\hat{\w}_i,\hat{\w}_{-i})$ will not result in an increase
        the utility
        of~$i$. %This completes the sufficiency proof. % of sufficiency of Properties~\ref{prop:wt_game_dom_agent}-~\ref{prop:wt_game_not_dom_agent}.
%%
%\margin{rephrase so that it looks similar to the proof of the previous lemma}
%\marginn{updated}
%%
%\mo{Finally, \labelcref{eq:equivalent}}
%The second part of the lemma 
%follows from the first part of this lemma, Lemma~\ref{lem:nash_partition_game}, and the definition in~\labelcref{eq:support_func}.
\end{proof}

As a direct implication of Lemma~\ref{lem:nash_weight_game}, for any
$\W \in [0,1]^{n \times m}$ we can define
$\setfunc : [0,1]^{n \times m} \to (2^\tsk)^n$ as
\begin{align}
	\setfunc(\W) \ldef (\supp(\w_1),\cdots,\supp(\w_n)),
	\label{eq:setfunc}	
\end{align}
where
$\supp(\w_i) \ldef \{q \in \tsk \,|\, w^q_i = 1\},\,\, \forall i\in
\agt$.
%% 
%\margin{I've simplified the first line to agents, I think it is clear}
%\marginn{noted}
%%
%\begin{subequations}
%\begin{align}	
%	\label{eq:translated_support} & \supp(\w_i) \ldef \{q \in \tsk \,|\, w^q_i = 1\},\,\, \forall i\in \agt, \\% \w_i \in \{\w_j\}_{j \in \agt}\\
%	\label{eq:ne_support_set} & \setfunc(\W) \ldef (\supp(\w_1),\cdots,\supp(\w_n))\,.
%\end{align}
%\label{eq:support_func}
%\end{subequations}
%%
%\margin{what is $\w$ in the first line? is it $\W$? Is the 1 in the $\w^q = 1$ a vector of ones? Then please introduce a notation for it and add it to the notations section. If the $\supp$ applies to any vector, then we can use it in the second line as applied to each $\w_i$. If it isn't, then the second line is not well defined as it $\supp$ was not defined over vectors. Please fix. }
%\marginn{fixed}
%%
Then, 
\begin{align}
	\ne(\game_P) = \setfunc(\ne(\game_W)).
	\label{eq:equivalent}
\end{align}
Next, we %use the previous two results to
relate the optimal partition and the NE of the two games through the
following theorem.
%Later we show that when there is a unique dominating agent for each task, the optimizers and the NE are in fact unique.

\begin{theorem}[Optimal partitions and Nash equilibria]\label{th:opt_subset_ne}
%\thmtitle{Optimal partitions and Nash equilibria}
%Consider the set
Given the problem in \labelcref{eq:partition_optimization} , $\opt \subseteq \setfunc(\ne(\game_W))$, where
%\vspace{-.2cm}
\begin{align}
	\opt \ldef \{\salloc^* \,\,|\,\, \salloc^* \,\, \mathrm{is\,\, a\,\, solution\,\, to\,\,}\labelcref{eq:partition_optimization} \}.
	\label{eq:opt}
\end{align}
\end{theorem}
\begin{proof}
	By~\labelcref{eq:equivalent}, we can equivalently show that $\opt \subseteq \ne(\game_P)$. Let $\salloc^* \in \opt$. It is easy to see that $q \in \alloc^*_i$ only if $i \in \agt$ is a dominating agent for task $q$. Moreover, if $j \in \agt$ is not dominating for $q$, then $q \notin \alloc^*_j$. Then from Lemma~\ref{lem:nash_partition_game}, $\salloc^* \in \ne(\game_P)$. The rest follows from~\labelcref{eq:equivalent}.
\end{proof}

The above result states that if an agent $i \in \agt$ is assigned tasks using the translated support of the NE of $\game_W$, this set is a superset of the optimizers of~\labelcref{eq:partition_optimization}. %With a slight abuse of notation, we interpret this result as, the optimal partition is a subset of the NE of $\game_W$. 
The extra solutions arise when there are non-unique dominating agents for a task. When there are  unique dominating agents, the next result shows there is a unique NE for $\game_W$. This follows from Lemma~\ref{lem:nash_weight_game} immediately, so we skip a formal proof.

\begin{corollary}[Uniqueness of Nash equilibria] \label{col:nash_wt_game_unique_dominate}
%\thmtitle{Uniqueness of NE under unique dominating agents} 
%%
%\margin{I changed the title of this result, and replaced Nash equilibrium by NE. I suggest we use this widely employed acronym. }
%\marginn{noted and updated}
%%
Suppose that for each $q \in \tsk$, $i^*_q$ is the unique dominating agent for task $q$. %Moreover, suppose $\salloc^*$ is a solution to~\labelcref{eq:partition_optimization}. 
Then $\ne(\game_W) = \{\hat{\W}\}$ where, for each $q \in \tsk$, $\hat{\W}$ satisfies
%\begin{enumerate}
%	\item for each task $q \in \tsk$, $\hat{w}^q_{i^*_q} = 1$;
%	\item if $j \neq i^*_q$, then $\hat{w}^q_j = 0 $.
%\end{enumerate}
%\begin{center}
%\begin{tabular}{ll}
	1) $\hat{w}^q_{i^*_q} = 1$, and 2) $\hat{w}^q_j = 0$, $\forall \, j \neq i^*_q$. 
%\end{tabular}
%\end{center}
Further, $\salloc^* = \setfunc(\hat{\W})$ is the unique solution to~\labelcref{eq:partition_optimization}.
\proofend
\end{corollary} 

In general, $\ne(\game_W)$ is a superset of the set
of optimal partitions. The next example makes this clear.

%\begin{example}\label{ex:superset_structure}
%\thmtitle{Optimal partitions and Nash equilibria}
%Let $\agt = \{1,2\}$ and $\tsk = \{a,b\}$. Assume the values $f_1(a) = f_2(a) = 0.5$, $f_1(b) = 0.7$ and $f_2(b) = 0.3$. Then, the optimal partitions are $\opt = \{(\{a,b\},\varnothing),(\{b\},\{a\})\}$. The NE of the partition game are $\ne(\game_P) = \{(\{a,b\},\varnothing),(\{a,b\},\{a\}),(\{b\},\{a\})\}$. The NE of the weight game are  
%\begin{align*}
%	\ne(\game_W) = \left\{
%	\begin{bmatrix}
%		1 & 1\\
%		\lambda & 0 
%	\end{bmatrix},
%	\begin{bmatrix}
%		1 & 1\\
%		1 & 0
%	\end{bmatrix},
%	\begin{bmatrix}
%		\mu & 1\\
%		1 & 0
%	\end{bmatrix} \right\}.
%\end{align*}
%where $\lambda,\mu$ can independently take any value in $[0,1)$. Thus, in this case $\opt \subsetneq \ne(\game_P) = \setfunc(\ne(\game_W))$. Interestingly, note that there is an optimal partition in which agent $2$ does not get any task. %So in general, not all agents can be assigned tasks. 
%\bulletend
%\end{example}
\begin{example}[Optimal partitions and Nash equilibria]\label{ex:superset_structure}
%\thmtitle{Optimal partitions and Nash equilibria}
Let $\agt = \{1,2\}$ and $\tsk = \{\tskind{a},\tskind{b}\}$. Assume the values $f_1(\tskind{a}) = f_2(\tskind{a}) = 0.5$, $f_1(\tskind{b}) = 0.7$ and $f_2(\tskind{b}) = 0.3$. Then, $\opt = \{(\{\tskind{a},\tskind{b}\},\varnothing),(\{\tskind{b}\},\{\tskind{a}\})\}$; $\ne(\game_P) = \{(\{\tskind{a},\tskind{b}\},\varnothing),(\{\tskind{a},\tskind{b}\},\{\tskind{a}\}),(\{\tskind{b}\},\{\tskind{a}\})\}$; and  
\begin{align*}
	\ne(\game_W) = \left\{
	\begin{bmatrix}
		1 & 1\\
		\lambda & 0 
	\end{bmatrix},
	\begin{bmatrix}
		1 & 1\\
		1 & 0
	\end{bmatrix},
	\begin{bmatrix}
		\mu & 1\\
		1 & 0
	\end{bmatrix} \right\},
\end{align*}
where $\lambda,\mu$ can independently take any value in $[0,1)$. Thus, in this case $\opt \subsetneq \ne(\game_P) = \setfunc(\ne(\game_W))$. Interestingly, note that there is an optimal partition in which agent $2$ does not get any task. %So in general, not all agents can be assigned tasks. 
\bulletend
\end{example}

Next, we design a dynamical system using which the agents can figure
out the optimal partition on their own.

\section{Best Response Projected Gradient Ascent} \label{sec:grad_ascent}

From the previous section, we know that if the agents play the weight
game $\game_W$, then the NE form a superset of the
optimal task partition (with slight abuse of notation). Thus, here we
let the agents update their weights (from any initial feasible weight)
using the gradient of their utility while assuming the others do not
change their weights. For such a dynamical system, we aim to relate
its equilibria to the NE of $\game_W$ and hence also
relate it to the set $\opt$ of optimal solutions
to~\labelcref{eq:partition_optimization}. Now, from~\labelcref{eq:agent_reward_max}, it can be seen that, 
\begin{align}
	\frac{\partial}{\partial w^q_i}U_i = f_i(q) - \max_{j \in \agt \setminus \{i\}}f_j(q)w^q_j \, \rdef u^q_i (\w^q)\,.
	\label{eq:partial_ui} 
\end{align}
Thus, the weights are updated using the following dynamics:
\begin{align}
	w^q_i(t+1) = \big[ w^q_i(t) + \gamma^q_i \, u^q_i(\w^q(t)) \big]_0^1, %\,\,\forall i \in \agt, \forall q \in \tsk,
	\label{eq:weight_dynamics_max}
\end{align}
with $\gamma^q_i \in \real_{>0}$,
$\forall i \in \agt, \forall q \in \tsk$. 
%
%\margin{explain the 0,1
%  notation of the dynamics. I've removed the forall i and q inside the
%  equation, as we say it in this first line}
%\marginn{The $[\cdot]_0^1$ function is mentioned in the notations}
%
%We call this the best response projected gradient ascent dynamics (\dynacr).
We call this the projected best response ascending gradient dynamics (\dynacr).
%%
%\margin{can you think of an acronym that can actually be pronounced?}
%\marginn{updated}
%%
From~\labelcref{eq:partial_ui} and~\labelcref{eq:weight_dynamics_max}, note
that in order to compute the weight updates, each agent $i \in \agt$
needs to know $f_j(q)w^q_j$, for all $j \neq
i$. This requires that each agent must talk to every
other agent  to compute its own gradient.  The equilibrium
points of this dynamics is given by

\vspace{-.3cm}
\begin{align}
	\notag & \eqpt \ldef  \Big\{ \W \in \kcell^{n \times m} \ \vline \  \Big[w^q_i + \gamma^q_i \, u^q_i(\w^q) \Big]_0^1 = w^q_i,\\ 
	& \hspace{15em} \forall i \in \agt, \forall q \in \tsk \Big\}.
	\label{eq:eqpt}
\end{align}
For this, the following result can be stated immediately.

\begin{lemma}[Equilibrium weights are Nash equilibria] \label{lem:max_dyn_eq}
The weight matrix $\bar{\W} \in \eqpt$ if and only if $\bar{\W} \in \ne(\game_W)$.
\end{lemma}
\begin{proof}
  We prove this by showing that $\bar{\W} \in \eqpt$ if and only if
  $\Wb$ follows Properties~\ref{prop:wt_game_dom_agent}
  and~\ref{prop:wt_game_not_dom_agent} of
  Lemma~\ref{lem:nash_weight_game}.
%First we show necessity of the properties.
Suppose that $\bar{\W} \in \eqpt$ and consider an arbitrary but fixed $q \in \tsk$. 
%Then property (1) holds from the dynamics in~\labelcref{eq:weight_dynamics_max} and the definition in~\labelcref{eq:eqpt}. 
We prove Property~\ref{prop:wt_game_dom_agent} by contradiction and
assume that $\forall i \in \agt$ dominating for $q$, $\bar{w}^q_i <
1$. Now, for any dominating agent $i^*_q \in \agt$,
$\max_{j \in \agt \setminus \{i^*_q\}}f_j(q)\wbar^q_j < f_{i^*_q}(q)$.
%This comes from the assumption that $\bar{w}^q_i < 1$, $\forall i \in \agt$ such that $i$ is a dominating agent for task $q$ (and the fact that $f_j(q)w^q_j < f_{i^*_q}(q)$, $\forall w^q_j \in [0,1]$ if $j$ is not a dominating agent). 
Thus, $u^q_{i^*_q}(\wb^q) > 0$. Since $\wbar^q_{i^*_q} < 1$ and
$\gamma^q_{i^*_q} > 0$, this contradicts $\Wb \in \eqpt$.
%This contradicts the assumption and proves necessity of property (2).

Next we prove Property~\ref{prop:wt_game_not_dom_agent} also by
contradiction. Suppose that $\exists \,j \in \agt$ not dominating for
task $q$ but $\bar{w}^q_j > 0$. Due to
Property~\ref{prop:wt_game_dom_agent}, let $i^*_q \in \agt$ be the
dominating agent for task $q$ such that $\wbar^q_{i^*_q} = 1$. Then,
\begin{align*}
  u^q_j(\wb^q) = f_j(q) - \max_{k \in \agt \setminus \{j\}} f_k(q) \wbar^q_k = f_j(q) - f_{i^*_q}(q) < 0\,.
\end{align*}
Again, as $\bar{w}^q_j > 0$ and $\gamma^q_j > 0$, this contradicts
$\Wb \in
\eqpt$. %This conradicts the assumption and proves necessity of property (3). This also completes the proof of necessity.

To show sufficiency, let $\Wb$ satisfy
Properties~\ref{prop:wt_game_dom_agent}
and~\ref{prop:wt_game_not_dom_agent}. Then it is easy to see that for
each task $q \in \tsk$, $u^q_i(\wb^q) \geq 0$ if $i \in \agt$ is dominating
for $q$ with $\wbar^q_i = 1$, $u^q_i(\wb^q) = 0$ if $i \in \agt$ is dominating
for $q$ with $\wbar^q_i < 1$, and $u^q_j(\wb^q) < 0$ if $j \in \agt$ is not dominating for $q$ (hence $\wbar^q_j = 0$). Then, $\Wb \in \eqpt$ follows since $\gamma^q_j >
0$. %This completes the proof.
\end{proof}

From Lemma~\ref{lem:max_dyn_eq}, we can also infer that if there is a
unique dominating agent, then the equilibrium set becomes a singleton
and follows the same structure as in
Corollary~\ref{col:nash_wt_game_unique_dominate}.

In what follows, we show that starting from any initial weights, the
dynamics~\labelcref{eq:weight_dynamics_max} converges to an equilibrium.

\begin{theorem}[\dynacr\, converges to an equilibrium weight]\label{th:wt_converge}
%\thmtitle{\dynacr converges to equilibrium weight} 
%%
%\margin{shorter title? titles for all defs/assumptions/result look good, but hopefully we can make them short (...)}
%\marginn{I have renamed the dynamics with an acronym}
%%
Suppose Assumption~\ref{asmp:not_all_dominating} holds.
  Consider the dynamics~\labelcref{eq:weight_dynamics_max} with an initial
  condition $\W(0) \in \kcell^{n \times m}$ and let $\W(t)$ be the 
  solution trajectory. Then $\lim_{t \to \infty} \W(t) = \Wb \in \eqpt$. %Then $\lim_{t \to \infty} w^q_i(t) = \bar{w}^q_i$, $\forall i \in \agt$, $\forall q \in \tsk$ with $\Wb \in \eqpt$.
\end{theorem}

\begin{proof}
  Notice that for the dynamics~\labelcref{eq:weight_dynamics_max}, the
  weight associated with each task evolves independently from the
  weights associated with other tasks. Thus, consider an arbitrary but
  fixed $q \in \tsk$. Next, consider any $i \in \agt$ that is dominating for
  $q$. From~\labelcref{eq:partial_ui}, $u^q_i(\w^q) \geq 0$,
  $\forall \, \w^q \in [0,1]^n$. Thus, since $\gamma^q_i > 0$,
  $\forall i \in \agt$, $w^q_i(t)$ is
  non-decreasing.
  Hence, $w^q_i(t) \to \what^q_i \in [0,1]$ as
  $t \to \infty$, since $[0,1]$ is compact. Now consider $\I^q \ldef \{j \in \agt\,|\, j \,\, \mathrm{is \,\, dominating \,\, for}\,\, q\}$, the set $\X \ldef \{\v \in [0,1]^{|\I^q|} \,|\, v_j = 1\,\, \mathrm{for\,\, some}\,\, j \in \I^q \}$ and define the continuous function $V(\w^q) \ldef d(\{w^q_i\}_{i \in \I^q}, \X)$. 
It is clear that %$d(\{w^q_i(t+1)\}_{i \in \I^q}, \X) \leq d(\{w^q_i(t)\}_{i \in \I^q}, \X)$,
$V(\w^q(t+1)) \leq V(\w^q(t))$ $\forall t \in \intpos$. Applying the LaSalle invariance principle, there is convergence to the largest invariant set in $V(\w(t)) = V(\w(t+1))$ for all $t$. We argue this set is necessarily $\X$. Otherwise, invariance implies that $u^q_i(\w^q) = 0$ for any dominating agent $i \in \agt$. However, this occurs if and only if
%
%\margin{we should use $\hat{\w}_q$ here right?} 
%
  $\exists \,i' \in \agt$, $i\neq i'$, another dominating agent for
  task $q$ such that $w^q_{i'} = 1$; 
%%  
%\margin{with hats?}
%\marginn{this holds in general}
%%  
%%  
%\margin{swap indices}
%\marginn{done}
%% 
  otherwise, $u^q_i(\w^q) > 0$. %Thus $\X$ is the only invariant set under~\labelcref{eq:weight_dynamics_max} and by LaSalle's invariance principle, 
  Thus, $\{w^q_i(t)\}_{i \in \I^q} \to \X$ as $t \to \infty$ . This along with previous discussion proves that $\what^q_i$ follows Property~\ref{prop:wt_game_dom_agent}
%
%\margin{use label to refer to properties. But I think this refers to
 %   property 1. Let's add a bit more detail: suppose that, for all
   % dominating agents $\what_i^q <1$; then, by (9), $u_i^q (\w^q)>0$
  %  which is a contradiction. Therefore property 1 must hold. }
%\marginn{do we still have to add more? That is the point of the iff statement}
%
  of Lemma~\ref{lem:nash_weight_game}.  

  Next consider any $j \in \agt$ that is not dominating for $q$. From the previous part of the proof, we know that there is
  a $i \in \agt$ dominating for $q$ for which
  $w^q_i(t) \to 1$ and thus $f_i(q)w^q_i(t) \to f_i(q)$ as
  $t \to \infty$. 
%%
%\margin{How does such an $i'$ exist? We shown that weights will increase to a value in $[0,1]$, and that they will be 0 if there is an agent for which the weights become one. But we haven't shown there is one weight that will become one. What if there are more than one dominating agent for task $q$? We haven't
%made assumptions about this.}
%\marginn{Updated with La Salle argument}
%%
This implies that
  $\exists \,\tau \in \intpos$ such that
  $\max_{k \in \agt \setminus \{j\}}f_k(q)w^q_k(t) \geq f_j(q) +
  \nu$, for some $\nu > 0$ and $\forall t \geq
  \tau$. Then, as $u^q_j(\w^q(t)) \leq -\nu < 0$, $w^q_j(t)$ is a strictly decreasing sequence (after $\tau$
  time steps). Thus, from the dynamics
  in~\labelcref{eq:weight_dynamics_max}, %we see that
  $w^q_j(t) \to \what^q_j = 0$ as $t \to \infty$. Hence, $\what^q_j$
  follows Property~\ref{prop:wt_game_not_dom_agent}
%%  
%\margin{there is no property (3)
%    anymore, so use labels to avoid problems}
%\marginn{updated}
%%
of Lemma~\ref{lem:nash_weight_game}.
%This completes the proof.
\end{proof}

When there is a unique dominating agent for each task, we can
guarantee finite-time convergence to an optimal
partition.% We state this formally next. %The proof uses similar ideas
%from the previous proof.

\begin{theorem}[\dynacr\, converges in finite
  time]\label{th:finite_time_wt_converge}
  % \thmtitle{\dynacr for unique dominating agent converges in finite
  % time} Let $f_i(q) \geq 0$, $\forall q \in \tsk$. Suppose
  % $\nexists \, q \in \tsk$ such that $f_i(q) = 0$,
  % $\forall i \in \agt$ and that $\forall q \in \tsk$, there exists a
  % dominating agent. Consider the
  % dynamics~\labelcref{eq:weight_dynamics_max} with an initial condition
  % $\{\w_i(0)\}_{i \in \agt} \in (\kcell^m)^n$. Then
  % $\lim_{t \to \infty} w^q_i(t) = \bar{w}^q_i$,
  % $\forall i \in \agt$, $\forall q \in \tsk$ with
  % $\{\wb^i\}_{i \in \agt} \in \eqpt$.
%
  % Moreover, $\exists \,T \in \intpos$ such that
  % $w^q_i(t) = \bar{w}^q_i$, $\forall i \in \agt$,
  % $\forall q \in \tsk$, $\forall t \geq T$.
%
Suppose Assumption~\ref{asmp:not_all_dominating} holds and
  suppose that for each $q \in \tsk$, there exists a unique dominating
  agent, $i_q^* \in \agt$.  Define
  $\gamma \ldef \min\limits_{i \in \agt, q\in \tsk} \gamma^q_i > 0$,
  and let
  $\delta \ldef \min\limits_{q \in \tsk} \big[ f_{i^*_q}(q) -
  \max\limits_{j \neq i^*_q} f_j(q) \big] > 0$.
%\begin{align*}
%	\gamma & \ldef \min_{i \in \agt, q\in \tsk} \gamma^q_i > 0;\\
%	\delta & \ldef \min_{q \in \tsk} \, \Big[ f_{i^*_q}(q) - \max_{j \neq i^*_q} f_j(q) \Big] > 0\,.
%\end{align*}
%\begin{subequations}
%\begin{align}
%%	\label{eq:delta_q} \delta_q & \ldef f_{i^*_q}(q) - \max_{j \neq i^*_q} f_j(q) > 0,\\
%	\label{eq:delta} \delta & \ldef \min_{q \in \tsk} \Big[ f_{i^*_q}(q) - \max_{j \neq i^*_q} f_j(q) \Big] > 0,\\
%	\label{eq:alpha} \gamma & \ldef \min_{i \in \agt, q\in \tsk} \gamma^q_i > 0 \,.
%\end{align}
%\label{eq:alpha_delta}
%\end{subequations}
  Consider the dynamics~\labelcref{eq:weight_dynamics_max}
  starting from $\W(0) \in \kcell^{n \times m}$, with the solution trajectory
  $\W(t)\rightarrow \Wb \in \eqpt$, as $t\rightarrow \infty$.  Then $w^q_i(t) = \bar{w}^q_i$,
  $\forall i \in \agt$, $\forall q \in \tsk$,
  $\forall t \geq 2\left\lceil(\gamma\delta)^{-1}
  \right\rceil$. %, with
 % $\Wb \in \eqpt$.
\end{theorem} 
\begin{proof}
  From Lemma~\ref{lem:max_dyn_eq} and
  Corollary~\ref{col:nash_wt_game_unique_dominate}, it is clear that
  $\eqpt$ is a singleton set. Let $ \Wb \in \eqpt$ be the unique
  equilibrium point. From Theorem~\ref{th:wt_converge}, we know that
  $\W(t) \to \Wb$ as $t \to \infty$. % By hypothesis, for each
  % $q \in \tsk$, let $i^*_q$ be the dominating agent for task $q$.
  From~\labelcref{eq:partial_ui},
%  \begin{align*}
    $u^q_{i^*_q}(\w^q(t)) \geq \delta > 0, %\quad 
    \forall t \in \intpos$
%  \end{align*}
  and hence from~\labelcref{eq:weight_dynamics_max}, $\forall t\in \intpos$,
  %\begin{align*}
    $w^q_{i^*_q}(t+1)
    \geq [ w^q_{i^*_q}(t) + \gamma^q_{i^*_q} \, \delta ]_0^1 \geq [ w^q_{i^*_q}(0) + (t+1)\, \gamma^q_{i^*_q} \, \delta ]_0^1$.
  %\end{align*}
  The inequality holds since $[\cdot]_0^1$ is a nondecreasing function.
%%
%  \margin{This step is not so direct. We need some properties of the
%    projection operator. I think that $P(v) \le |v|$ for $v \in \real$,
%    where $P$ is the projection to any convex set. Then, the opposite
%    holds $P(P(v) + v') \le P(v +w)$. Can you elaborate on this? }
%  \marginn{Is the first statement true? For the real line convex sets are just intervals. Then suppose the convex set is $[1,2]$ and $v = 0.5$. Then $P(v) = 1 > 0.5$. The inequality holds since $[\cdot]_0^1$ is a non decreasing function}
%%
  Thus,
%  \begin{align*}
%    w^q_{i^*_q}(t) \geq \Big[ w^q_{i^*_q}(0) + t\, \gamma^q_{i^*_q} \, \delta_q \Big]_0^1, \quad \forall t \in \intpos
%  \end{align*}
  $w^q_{i^*_q}(t) = 1$, for all
%  \begin{align*}
%    \forall t \geq \left\lceil(\gamma\delta)^{-1} \right\rceil \geq \frac{1}{\gamma \delta} \geq \frac{1}{\gamma^q_{i^*_q} \delta} \geq \frac{1 - w^q_{i^*_q}(0)}{\gamma^q_{i^*_q} \delta}\,.
%  \end{align*}
%\begin{align*}
     $t \geq \left\lceil(\gamma\delta)^{-1} \right\rceil \geq (\gamma\delta)^{-1} \geq [1 - w^q_{i^*_q}(0)] [\gamma^q_{i^*_q} \delta]^{-1}$.
%\end{align*}
Now define $\tau \ldef \left\lceil(\gamma\delta)^{-1} \right\rceil$ and
consider any $j \neq i^*_q$ and notice from~\labelcref{eq:partial_ui} that
%%
%\margin{If you are not reusing an equation, please use * to remove
%  extra numbers. The following 2 are examples.  In fact, it doesn't
%  make sense to add labels inside proofs: proofs are skipped most of
%  the time}
%\marginn{updated and noted}
%%
%\begin{align*}
  $u^q_j(\w^q(t)) \leq -\delta < 0, %\quad 
  \forall t \geq \tau$.
%\end{align*}
Thus, $w^q_j(t) \leq [ w^q_j(\tau) - t\, \gamma^q_j \, \delta ]_0^1, \forall t \geq \tau.$
%%
%\margin{this inequality may work because of negativity, but
%  likewise, you need to discuss it as it's not obvious. }
%\marginn{similar reason}
%%
%\begin{align*}
%	w^q_j(t) \leq \Big[ w^q_j(\tau) - t\, \gamma^q_j \, \delta \Big]_0^1, \quad \forall t \geq \tau\,.
%\end{align*}
The inequality again holds since $[\cdot]_0^1$ is non decreasing. So, $w^q_j(t) = 0$, for all
%\begin{align*}
	$t \geq \tau + \left\lceil(\gamma \delta)^{-1}\right\rceil \geq \tau + [\gamma^q_j \delta ]^{-1} \geq \tau + w^q_j(\tau) [\gamma^q_j \delta]^{-1}$.	
%\end{align*}
%This completes the proof.
\end{proof}

%%
%\margin{for remark envs, I prefer to use the straight fonts}
%\marginn{noted}
%%
\begin{remark}[On the effect of step-size on convergence]
  \label{rem:finite_time_two_steps} {\rm By
    Theorem~\ref{th:wt_converge},~\labelcref{eq:weight_dynamics_max}
    converges to an equilibrium weight when $\gamma^q_i > 0$,
    $\forall i \in \agt$, $\forall q \in \tsk$. Thus agents can choose
    any constant positive step size and guarantee convergence to a NE
    of the weight game. Further inspection of
    Theorem~\ref{th:finite_time_wt_converge} leads to this interesting
    observation. Since $\delta^{-1} > 0$, the individual
    $\gamma^q_i$'s can be chosen in such a way that
    $0 < {(\gamma\delta)}^{-1} < 1$. Then
    $2\left\lceil(\gamma\delta)^{-1} \right\rceil = 2$. That is, by
    choosing a sufficiently large step size and communicating with
    every other agent, the agents can reach the NE in at most
    \emph{two} time steps. }
\bulletend
\end{remark}

In order to avoid all-to-all communication, it is possible to
adapt~\labelcref{eq:weight_dynamics_max}
%%
%\margin{refer to 10}
%\marginn{done}
%% 
introducing a consensus subroutine. %~\cite{RL-DV:15}.
%%
%\margin{let's leave it as this, and no citation.}
%\marginn{noted}
%%
%%
%\margin{maybe add a rfrence}
%\marginn{added}
%%
%This, however, introduces
%some issues when dealing with unknown rewards. Thus, we propose a
%distributed consensus-type of dynamics in
%%Section~\ref{sec:distributed_dyn} 
%the next section to tackle the problem directly.
In the next section, we utilize this idea to handle decentralization
together with unknown
rewards. % to propose a distributed version of the algorithm.

%%%%%%%%%%%%%%%%%%%%%%%%%%%
\section{Distributed Task Allocation} \label{sec:distributed_dyn}

Here, we provide a solution to
Problem~\ref{prob:main}~(\ref{prob:dist_ne_algo}). Recall that in
Section~\ref{sec:grad_ascent}, each agent $i \in \agt$ computes
$\max_{j \neq i}f_j(q)w^q_j$ using information from all other agents.
Here, we introduce a communication graph $\grph \ldef (\agt,\edg)$
with vertex set $\agt$. The arc set $\edg$ defines the connections
between agents, with $(i,j) \in \edg$ if and only if $i \in \agt$ can
send information to $j \in \agt$. For the sake of brevity, let
$\dg \ldef \diam(\grph)$.

For such a setup, the following result gives a way to find the max and
second unique max values in a distributed fashion. This is useful in
providing a distributed \dynacr\,(d-\dynacr).  
\begin{lemma}[Agreement on the two largest variables in a network]
  \label{lem:max_consensus} 
%%
%\margin{let's change this title, by max
%consensus everyone thinks of pure max consensus, which is known to
%converge in finite time. How about this?}
%\marginn{This is perfect}  
%%
Let $\grph$ be a
strongly connected graph and consider
\begin{subequations}
\begin{align}
	\label{eq:max_consensus_no_switch} & M^q_i(t+1) = \max_{j \in \neighb_i} M^q_j(t),\\
	\label{eq:submax_consensus_no_switch} & S^q_i(t+1) = \submax \Big\{\{S^q_j(t)\}_{j \in \neighb_i},M^q_i(t),v^q_i \Big\},
\end{align}
\label{eq:consensus_no_switch}
\end{subequations}
with initial condition $M^q_i(0) = S^q_i(0) = v^q_i \in \real_{\geq 0}$, $\forall i \in \agt$, $\forall q \in \tsk$. Then $\forall \, q \in \tsk$ and $ \forall \, i \in \agt$;
\begin{enumerate}
	\item \label{prop:max_converge} $M^q_i (t) = \max \{v^q_j\}_{j \in \agt}$, $\forall \, t \geq \dg$,
	\item \label{prop:submax_converge} $S^q_i (t) = \submax \{v^q_j\}_{j \in \agt}$, $\forall \, t \geq 2\,\dg$.
\end{enumerate}
\end{lemma}

\begin{proof}
  We show this for an arbitrary but fixed $q \in \tsk$. Let
  $i^*_q \in \argmax \{v^q_j\}_{j \in \agt}$. Then,
  from~\labelcref{eq:max_consensus_no_switch},
  $M^q_{i^*_q}(t) = v^q_{i^*_q}$, $\forall t \in \intpos$. Thus,
  $\forall \, i^*_q \in \argmax \{v^q_j\}_{j \in \agt}$,
  $M^q_j(t) = v^q_{i^*_q}$, $\forall j \in \neigh_{i^*_q}$,
  $\forall t \geq 1$. Continuing this argument inductively proves
  Property~\ref{prop:max_converge} since $\grph$ is strongly
  connected.

To show Property~\ref{prop:submax_converge}, we use Property~\ref{prop:max_converge}. Now let $i^*_q \in \argsubmax \{v^q_j\}_{j \in \agt}$. Then, from~\labelcref{eq:submax_consensus_no_switch}, $S^q_{i^*_q}(t) = v^q_{i^*_q}$, $\forall t \geq \dg$. Thus, similarly, $\forall \, i^*_q \in \argsubmax \{v^q_j\}_{j \in \agt}$, $S^q_j(t) = v^q_{i^*_q}$, $\forall j \in \neigh_{i^*_q}$, $\forall t \geq \dg + 1$. Again, continuing this argument proves Property~\ref{prop:submax_converge} since $\grph$ is strongly connected.
\end{proof}

To compute the gradient and update the weights $w^q_i$ simultaneously, we propose the following dynamics:
\begin{subequations}
\begin{align}
  \label{eq:dist_wt_update} & \hspace{-1.0ex} w^q_i(t+1) \hspace{-0.5ex} = \hspace{-0.5ex} \Big[ w^q_i(t) + \gamma^q_i (t) \Big(\fseq^q_i(t) - \frac{1}{2} \big( M^q_i(t) + S^q_i(t) \big) \Big)   \Big]_0^1 , \\
  \label{eq:max_consensus} & \hspace{-1.0ex} M^q_i(t+1) \hspace{-0.5ex} = \hspace{-0.5ex} \switch \Big(\max_{j \in \neighb_i} M^q_j(t), e^q_i(t+1), t+1, T \Big),\\
  \notag & \hspace{-1.0ex} S^q_i(t+1) = \switch \Big(\submax \Big\{\{S^q_j(t)\}_{j \in \neighb_i},M^q_i(t),e^q_i(t) \Big\}, \\
  \label{eq:submax_consensus} & \hspace{12em} e^q_i(t+1), t+1, T \Big),\\
  \label{eq:input_injection} & \hspace{-1.0ex} e^q_i(t+1) = \switch \Big(e^q_i(t),z^q_i(t+1),t+1,T \Big)\,,
\end{align}
\label{eq:dist_pbrag}
\end{subequations}
for some $T \in \realnonnegative$ and where $\switch$ is the switching function
\begin{align}
	\switch \Big(m, \fseq,  t, T  \Big) \ldef 
	\begin{cases}
          \fseq, &\mathrm{if}\,\, t \mod T = 0,\\
          m, &\mathrm{otherwise}\,.
	\end{cases}
	\label{eq:switch}
\end{align}

\begin{remark}[d-\dynacr\, with agreement and periodic input
  injection]\label{rem:dist_dyn} {\rm Note that
    $\forall \, i \in \agt$, $\forall \, q \in \tsk$, the weight
    update in~\labelcref{eq:dist_wt_update} uses the sequence
    $\{\fseq^q_i(t)\}_{t \in \intpos}$ and a time-varying step-size
    $\gamma^q_i(t)$ instead of $f_i(q)$ and a constant step-size
    $\gamma^q_i$; respectively, as
    in~\labelcref{eq:weight_dynamics_max}. The periodic switching function
    $\switch$ ensures that $e^q_i(t)$ holds the value $\fseq^q_i(t)$
    for every $T$ time-steps. This in turn
    allows~\labelcref{eq:max_consensus} and~\labelcref{eq:submax_consensus} to
    run an agreement subroutine as~\labelcref{eq:consensus_no_switch}
    every $T$ time-steps with $v^q_i = \fseq^q_i(kT)$, for
    $k \in \intpos$. Thus, at every time-step which is a multiple of
    $T$, each agent believes that its own value is the maximum and
    corrects this belief over the next $T-1$
    time-steps.} %This allows us to tackle the unknown reward scenario.
\bulletend
\end{remark}

\begin{theorem}[Asymptotic behavior of d-\dynacr] 
\label{th:dist_pbrag_asymp}
%%
%\margin{recall the assumptions of the theorem in the statement (you may use just the labels if there is no space)}
%\marginn{done}
%%
Suppose Assumptions~\ref{asmp:not_all_dominating} and~\ref{asmp:convering_seq_reward} hold.
Define
\begin{align}
  \Delta^q \ldef \Big(\max_{i \in \agt}
  f_i(q) - \min_{i \in \agt} f_i(q) \Big) > 0, \, \forall q \in \tsk\,.
	\label{eq:max_minus_min}
\end{align}
Consider any $\varepsilon \in (0,1)$. Suppose $\forall t \in \intpos$,
$\gamma^q_i(t) = \alpha^q_i$, with
$0 < \alpha^q_i \leq \varepsilon \, (2\,\dg \, \Delta^q)^{-1}$,
$\forall i \in \agt$, $\forall q \in \tsk$.  Next, define
$\alpha \ldef \min_{i \in \agt, q \in \tsk} \alpha^q_i$,
$\mu^q \ldef 0.5\,(\max \{f_i(q)\}_{i \in \agt} - \submax
\{f_i(q)\}_{i \in \agt})$, and
\begin{align}
	& \mu \ldef (1-\nu) \min_{q \in \tsk}\mu^q > 0,
	\label{eq:max_minus_submax} 	
\end{align}
with $\nu \in (0,1)$.  Further, suppose
$T > 2\,\dg + (\alpha\,\mu)^{-1} + 1$.  Let $\W(t)$ be the solution trajectory
to~\labelcref{eq:dist_pbrag} starting from $\W(0) \in [0,1]^{n \times
  m}$. Then $\exists \, \tau(\W(0)) \in \intpos$ such that
$\forall\, t \geq \tau$,
\begin{enumerate}
\item \label{prop:dist_pbrag_dom_agt} $w^q_i(t) = 1$ if $i \in \agt$
  is dominating for $q \in \tsk$;
\item \label{prop:dist_pbrag_not_dom_agt} $w^q_j(t) \leq \varepsilon$
  if $j \in \agt$ is not dominating for $q \in \tsk$.
\end{enumerate}
Thus $\setfunc(\W(t))$ converges in finite number of time steps.
\end{theorem}

\begin{proof}
  First note that the bounds on $\alpha^q_i$'s and $T$ are valid because of
  Assumption~\ref{asmp:not_all_dominating}. Then, we show the claims
  for an arbitrary but fixed $q \in \tsk$.
	
  Recall that because of Assumption~\ref{asmp:convering_seq_reward},
  $\exists \, \tau_0 \in \intpos$ such that
  $\forall t,t' \geq \tau_0$,
  $\fseq^q_i(t) - 0.5 \, (\fseq^q_i(t) + \fseq^q_j(t')) < \Delta^q$,
  $\forall \, i,j \in \agt$. Moreover $\tau_0$ can be chosen such that
  $\fseq^q_{i^*_q}(t) - 0.5 \, (\fseq^q_{i^*_q}(t) + \fseq^q_j(t'))
  \geq (1-\nu)\,\mu^q > 0$, for any $\nu \in (0,1)$, if
  $i^*_q \in \argmax_{i \in \agt} f_i(q)$ and $\forall j \in \agt$
  such that $j \notin \argmax_{i \in \agt} f_i(q)$.
	
  Now consider any $i^*_q \in \argmax_{i \in \agt} f_i(q)$ and any
  $\nu \in (0,1)$. Then from Remark~\ref{rem:dist_dyn} and from the
  previous discussion, it is clear that, $\forall t \geq \tau_0$,
\begin{align*}
  & \alpha^q_{i^*_q} \big(\fseq^q_{i^*_q}(t) - 0.5 \big( M^q_{i^*_q}(t) + S^q_{i^*_q}(t) \big) \big) \geq \alpha^q_i (1-\nu) \, \mu^q \geq \alpha^q_i\mu\,.
\end{align*} 
This proves Property~\ref{prop:dist_pbrag_dom_agt} of this theorem as $\alpha^q_i\mu > 0$.

Next consider any $j \notin \argmax_{i \in \agt} f_i(q)$. Note from
Remark~\ref{rem:dist_dyn} that $\forall t \geq \tau_0$ such that
$t \in \{kT+2\dg,\cdots,2kT-1\}$ for some $k \in \intpos$, $w^q_j(t)$
strictly decreases, since,
\begin{align*}
  \fseq^q_j(t) - 0.5 ( M^q_j(t) + S^q_j(t) ) \leq -(1-\nu)\mu^q < 0\,.
\end{align*} 
Consider any $t \geq \tau_0$ such that
$t \in \{kT+2\dg,\cdots,2kT-1\}$ with $k \in \intpos$. Then, since
$w^q_j(kT+2\dg-1) \leq 1$,
\begin{align*}
	w^q_j(t) \leq 1 - ((t \mod T) - 2\dg)\alpha \mu,
\end{align*}
and hence because of the bound on $T$, $w^q_j(2kT-1) = 0$. Finally,
consider any $t \geq \tau_0$ such that $t \in \{kT,\cdots,kT+2\dg-1\}$
with $k \in \intpos$. Then, since $w^q_j(kT-1) = 0$ (from previous
arguments), we have
$w^q_j(t) \leq (t \mod T) \, \alpha^q_i \, \Delta^q$. Thus combining
all these arguments proves Property~\ref{prop:dist_pbrag_not_dom_agt}.

The final claim follows from the previous ones
and~\labelcref{eq:setfunc}.
%% 
%\margin{How about $C(\W(t))$? note we have stated this converges in finite time, but I'm not sure about that}
%\marginn{In finite time, we have dominating agent go to 1 and all other agents bounded by $\varepsilon < 1$. This means that $\supp(.)$ gives a constant output after that finite time.}
%%
\end{proof}

Note that the previous result does not guarantee that the weights
converge. This stems from the fact that at periodic times, each agent
believes that it gets the maximum reward for each
task. %But, we can eventually bound the growth of weights for non-dominating agents by any arbitrary $\epsilon \in (0,1)$.
Moreover, the previous result needs information about the limits of
the converging sequences to provide bounds for the step sizes and the
period of input injection. This can be avoided by allowing
time-varying step sizes as stated next.

\begin{theorem}[d-\dynacr\, converges to Nash equilibrium]
\label{th:dist_pbrag_converge}
Suppose Assumptions~\ref{asmp:not_all_dominating} and~\ref{asmp:convering_seq_reward} hold.
Let $\W(t)$ be the solution trajectory of~\labelcref{eq:dist_pbrag} from $\W(0) \in [0,1]^{n \times m}$, with $T > 2\,\dg + 1$ and $\forall i \in \agt$, $\forall q \in \tsk$,
\begin{align*}
	\gamma^q_i(t) = 
	\begin{cases}
		\alpha^q_i(k)>0, & \mathrm{if}\, t \in \{kT,\cdots,kT+2\dg-1\},\\
		\beta^q_i(k)>0, & \mathrm{if}\, t \in \{kT+2\dg,\cdots,2kT-1\},
	\end{cases}
\end{align*}
 with $k \in \intpos$. Further,
$\forall i \in \agt$, $\forall q \in \tsk$; take sequences
$\alpha^q_i(k) \to 0$ as $k \to \infty$ and $\beta^q_i(k) \to \infty$
as $k \to \infty$. Then $\W(t) \to \Wb \in \ne(\game_W)$ as $t \to \infty$.
\end{theorem}

\begin{proof}
  From hypothesis, $\forall \varepsilon > 0$,
  $\exists\, K \in \intpos$, such that
  $\forall t \in \{kT,\cdots,kT+2\dg-1\}$, with $k \geq K$,
  $\alpha^q_i(t) \leq \varepsilon \, (2\,\dg \, \Delta^q)^{-1}$, with
  $\Delta^q$ as in~\labelcref{eq:max_minus_min}. Moreover, $K$ can be
  chosen such that $\forall t \in \{kT+2\dg,\cdots,2kT-1\}$, with
  $k \geq K$,
  $T > 2\dg + ((1-\nu)\mu\min_{i \in \agt, q \in \tsk} \alpha^q_i
  \beta^q_i(t))^{-1} + 1$ for any $\nu \in (0,1)$ and with $\mu$ as
  in~\labelcref{eq:max_minus_submax}. Then the claim of this result is a
  consequence of applying similar arguments as in the proof of
  Theorem~\ref{th:dist_pbrag_asymp} and using~\labelcref{eq:equivalent}.
\end{proof}

We conclude this section by discussing some interesting observations
about the parameters in~\labelcref{eq:dist_pbrag}.

\begin{remark}[On the implementation of d-\dynacr]
  \label{rem:dpbrag_implementation} {\rm Note that even though
    $\diam(\grph)$ is an internal property of the communication graph
    $\grph$ and requires some structural knowledge of the same, the
    claims in Theorems~\ref{th:dist_pbrag_asymp}
    and~\ref{th:dist_pbrag_converge} remain true if $\dg$ is replaced
    with $n$. This is because $\diam(\grph) \leq n$. Moreover, these
    results can be extended to time-varying communication graphs with
    periodic connectivity because the agreement subroutine still works.
    Further, note that the conditions in
    Theorem~\ref{th:dist_pbrag_converge} are only sufficient for
    convergence. In fact, $\beta^q_i(k)$ need not grow unbounded, but
    then knowledge of converging reward values are required for proper
    functioning of the algorithm. For example, if $\mu$ is large, small
    values of $\beta^q_i(k)$ are sufficient to guarantee convergence;
    but if $\mu$ is small then $\beta^q_i(k)$ values have to be
    sufficiently large in order to guarantee that non-dominating
    agents are not assigned the task. Finally, notice that in order
    for the algorithm to work, each agent $i \in \agt$ has to pass
    two values ($M^q_i(t)$, $S^q_i(t)$) for each task $q \in \tsk$ to
    its neighbors at each time step. This makes the local
    communication cost of this algorithm of the order of $O(m)$ per
    iteration time.}  \bulletend
\end{remark}
%%
%\margin{read up to here}
%%

\begin{table}%[htbp]
\begin{center}
\caption{Approximate $f_i(q)$ values for Simulations}
\label{tab:sim_val}
\begin{tabular}{|l||*{4}{c|}}\hline
\backslashbox{$i \in \agt$}{$q \in \tsk$} & 1 & 2 & 3 & 4\\
\hline \hline
1 & 0.4536 & \cellcolor{teal!25}0.4407 & 0.2881 & 0.0055 \\ \hline
2 & 0.7504 & 0.2228 & 0.0411 & \cellcolor{teal!25}0.2801 \\ \hline
3 & \cellcolor{teal!25}0.7656 & 0.0987 & 0.1381 & 0.2491 \\ \hline 
4 & 0.3023 & 0.2211 & \cellcolor{teal!25}0.3334 & 0.2462 \\ \hline 
\end{tabular}
\end{center}
\begin{center}
\begin{tabular}{|l||*{4}{c|}}\hline
\backslashbox{$i \in \agt$}{$q \in \tsk$} & 5 & 6 & 7 & 8\\
\hline \hline
1 & 0.0049 & 0.2394 & \cellcolor{teal!25}0.3152 & 0.2217 \\ \hline
2 & 0.2374 & 0.0768 & 0.0852 & 0.1760 \\ \hline
3 & 0.2969 & 0.1003 & 0.1471 & \cellcolor{teal!25}0.6902 \\ \hline 
4 & \cellcolor{teal!25}0.3033 & \cellcolor{teal!25}0.4991 & 0.1231 & 0.5931 \\ \hline 
\end{tabular}
\end{center}
\end{table}

\begin{figure}
\begin{center}
	\begin{tabular}{c}
		\hspace{-2ex}\includegraphics[scale=0.28]{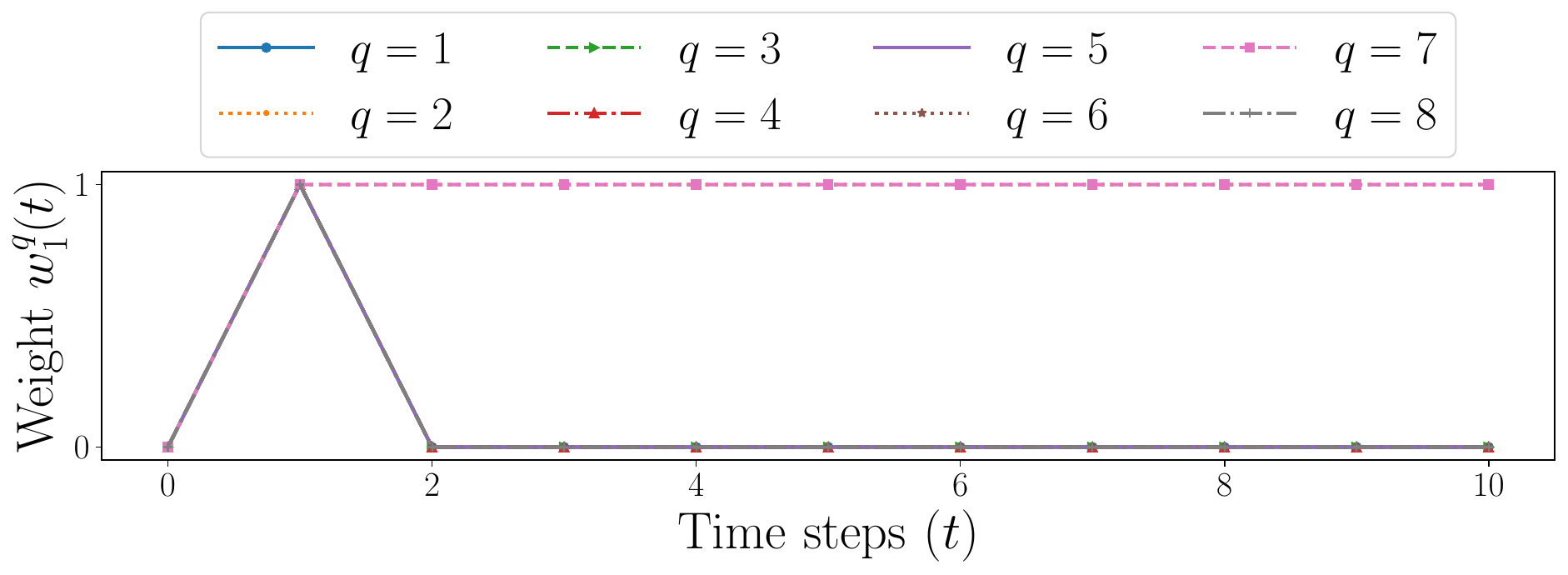} \\
		\hspace{-2ex}\includegraphics[scale=0.28]{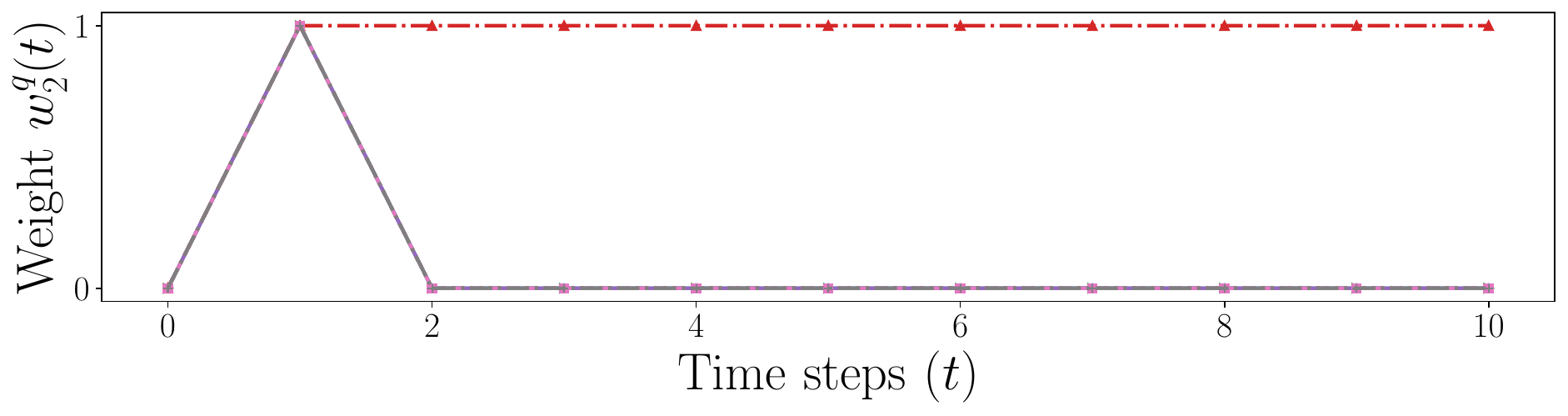} \\
		\hspace{-2ex}\includegraphics[scale=0.28]{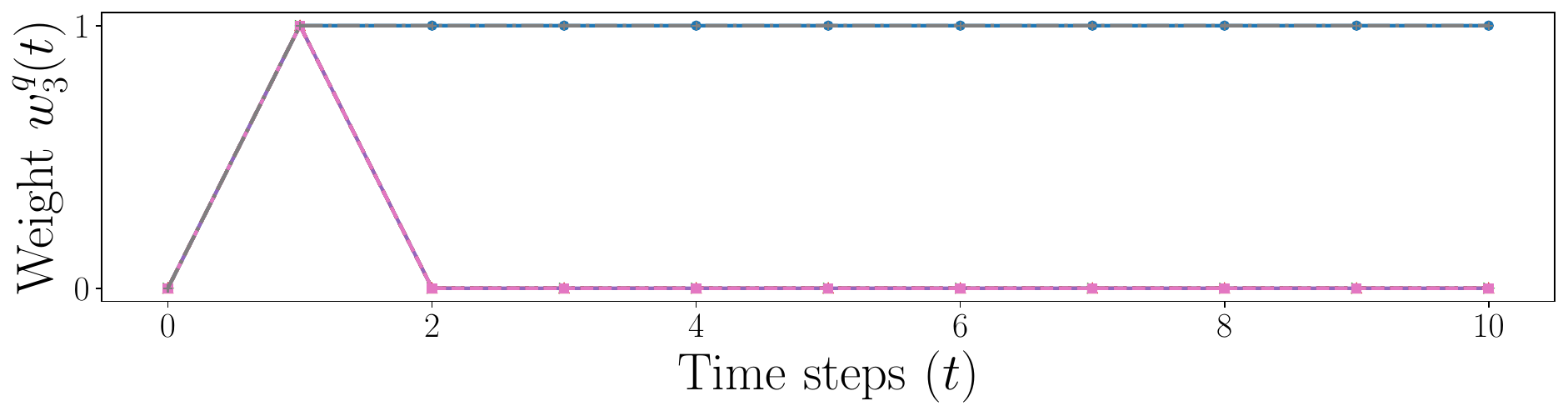} \\
		\hspace{-2ex}\includegraphics[scale=0.28]{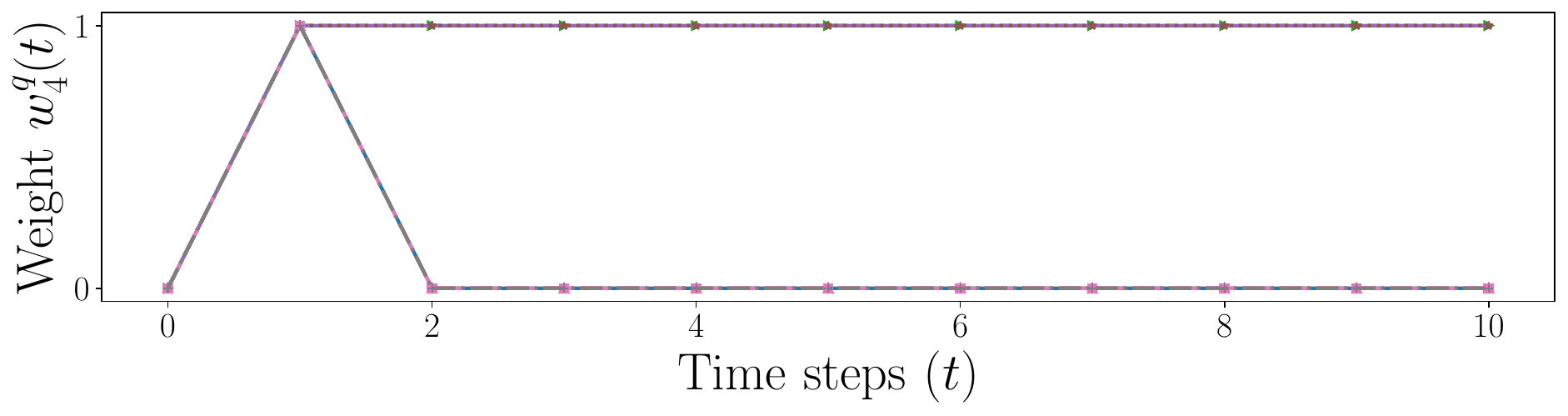}
	\end{tabular}
\end{center}
	\caption{\dynacr\, using~\labelcref{eq:weight_dynamics_max} and large step size $\gamma$. Plots share a common legend.}
	\label{fig:full_comm_high_alpha}
\end{figure}

\section{Simulations} \label{sec:sims}

In this section, we verify our major claims and illustrate some
interesting features of our algorithms.

%\subsection{Fast convergence with large $\gamma^q_i$} \label{sec:sims_high_alpha}
\subsection{Fast convergence of \dynacr\, with large step-size} \label{sec:sims_high_alpha}

Here, we simulate $n = 4$ agents to optimally allocate $m = 8$ tasks
with $\rf_i(q), \phi_i(q) \sim \unif[0,1]$,
$\forall i \in \agt, q \in \tsk$.  In particular,
Table~\ref{tab:sim_val} gives the approximate values of $f_i(q)$,
$\forall i \in \agt$, $q \in \tsk$. For each $q \in\tsk$, the
highlighted cell represents $\max_{i \in \agt} f_i(q)$.

We first verify the claim in Remark~\ref{rem:finite_time_two_steps}.
%Section~\ref{sec:grad_ascent}.  
Figure~\ref{fig:full_comm_high_alpha} shows the solution evolution
using~\labelcref{eq:weight_dynamics_max} from an initial $\W(0) =
\zero$. Optimal partition as in Figure~\ref{fig:full_comm_high_alpha}
is given as $\alloc_1 =\{2,7\}$, $\alloc_2 =\{4\}$,
$\alloc_3 =\{1,8\}$, $\alloc_4 =\{3,5,6\}$.  Here, since the values of
$f_i \in [0,1]$, $\gamma^q_i \in O(10^6)$ was required to make the
solutions converge in \emph{two} time steps. For larger deviations in
the values of $f_i$, much smaller values of $\gamma^q_i$'s can achieve
similar effects.

\subsection{Effect of constant step-size on d-\dynacr}
\label{sec:sims_dist_const_ss}
Here, we deal with the claims in Theorem~\ref{th:dist_pbrag_asymp} for
$n = 8$ agents optimally allocating $m = 1$ task. We take
$f_1(1) = B$, $f_2(1) = 0.9\,B$, and $f_i(1) = 0.3 \, B/ i$,
$\forall \, i \in \{3,\cdots,8\}$, with $B = 1000$. Thus agent $1$ is
the dominating agent. Further, we consider an unknown reward structure
with $\fseq^q_i(t) = f_i(q) + a^q_i\,\cos(b^q_i\,t)\exp(-c^q_i\,t)$,
$\forall i \in \agt$, $\forall q \in \tsk$, where
$a^q_i \sim \unif[0,f_i(q)]$, $b^q_i \sim \unif[0,10]$, and
$c^q_i \sim \unif[0,1]$. We set the communication graph
$\grph = (\agt,\edg)$ with $\edg = \{(1,2),(2,3),(3,4),(4,1)\}$.

Figure~\ref{fig:dist_const_ss} shows the solution evolution
using~\labelcref{eq:dist_pbrag} with constant step-size from an initial
$\W(0) = \zero$. It is interesting to note from
Figure~\ref{fig:dist_const_ss} that if $\varepsilon$ is large, then
$w^1_1(t)$ reaches $1$ faster, but the weights of the non-dominating
agents rise higher. On the other hand if $\varepsilon$ is small, then
the rise in the weights of the non-dominating agents is less but
$w^1_1(t)$ reaches $1$ slower. This is because $\varepsilon$ affects
the choice of $T$ as well.

\subsection{d-\dynacr\, with time-varying
  step-sizes} \label{sec:sims_distributed} Here, we again simulate
$n = 4$ agents optimally allocating $m = 8$ tasks. We take the unknown
reward structure as in Section~\ref{sec:sims_dist_const_ss} with
$f_i(q)$ as in Table~\ref{tab:sim_val}. Further, we use the
distributed approach using~\labelcref{eq:dist_pbrag} with time-varying
step-sizes as described in Theorem~\ref{th:dist_pbrag_converge}. We
also set the communication graph $\grph$ as in
Section~\ref{sec:sims_dist_const_ss}.
%%
%\margin{What is the effect of the buffer on the algorithm? }
%\marginn{Statement of the theorem has been updated} 
%%  

Figure~\ref{fig:dist_comm} shows the solution evolution
using~\labelcref{eq:dist_pbrag} from an initial $\W(0) = \zero$. Optimal
partition as in Figure~\ref{fig:dist_comm} is given as
$\alloc_1 =\{2,7\}$, $\alloc_2 =\{4\}$, $\alloc_3 =\{1,8\}$,
$\alloc_4 =\{3,5,6\}$. This is exactly same as the observation in
Section~\ref{sec:sims_high_alpha}.  Further, notice from
Figure~\ref{fig:dist_comm} that the weights of agents $3$ and $4$ take
longer time to settle than agents $1$ and $2$. In general, convergence
rate of the algorithm depends on the properties of the unknown reward
sequences and hence is difficult to characterize.

\begin{figure}
\begin{center}
	\begin{tabular}{c}
		\hspace{-2ex}\includegraphics[scale=0.28]{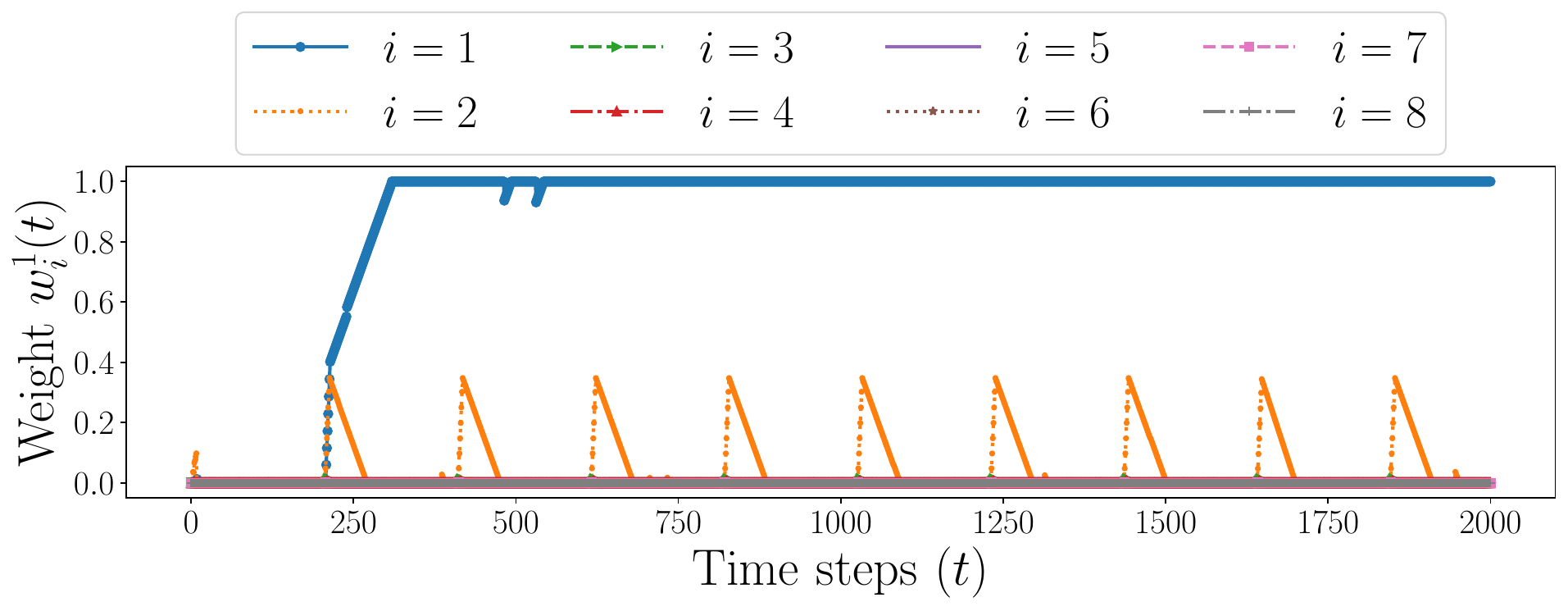} \\
		\hspace{-2ex}\includegraphics[scale=0.28]{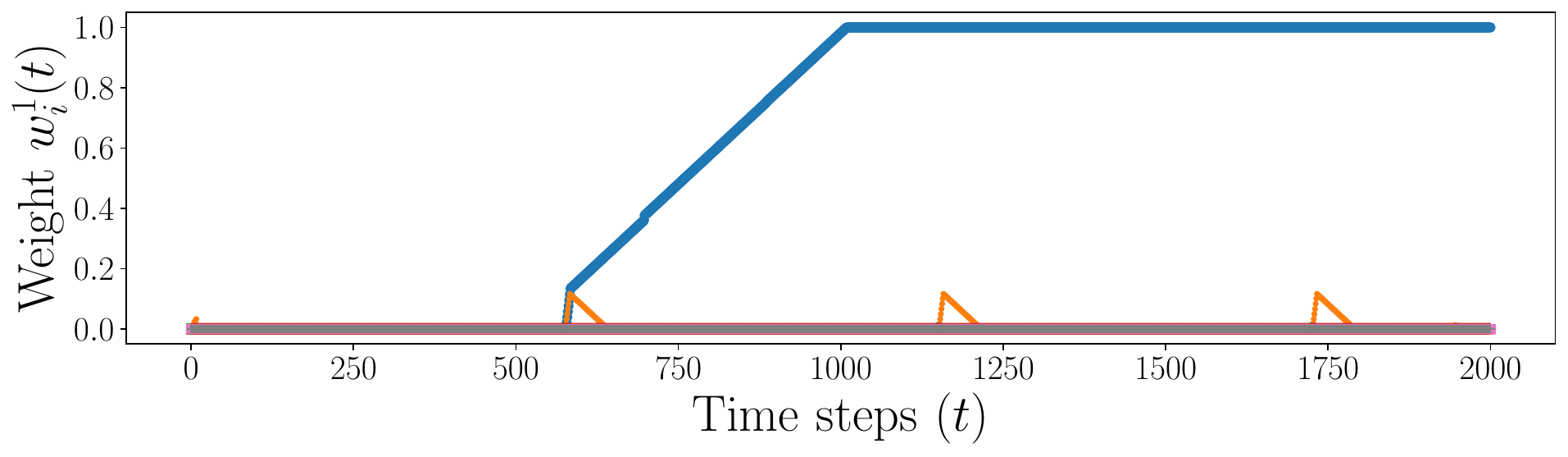}\\
	\end{tabular}
\end{center}
\caption{d-\dynacr\, for unknown rewards with constant step-size
  using~\labelcref{eq:dist_pbrag} on cyclic communication graph. Step-size
  $\gamma^q_i(t)$ and time-period $T$ were chosen as in
  Theorem~\ref{th:dist_pbrag_asymp} with different $\varepsilon$ and
  $\nu = 0.1$. The plots share a common legend. (Top)
  $\varepsilon = 0.9$. (Bottom) $\varepsilon = 0.3$.}
	\label{fig:dist_const_ss}
\end{figure}

\begin{figure}
\begin{center}
	\begin{tabular}{c}
		\hspace{-2ex}\includegraphics[scale=0.28]{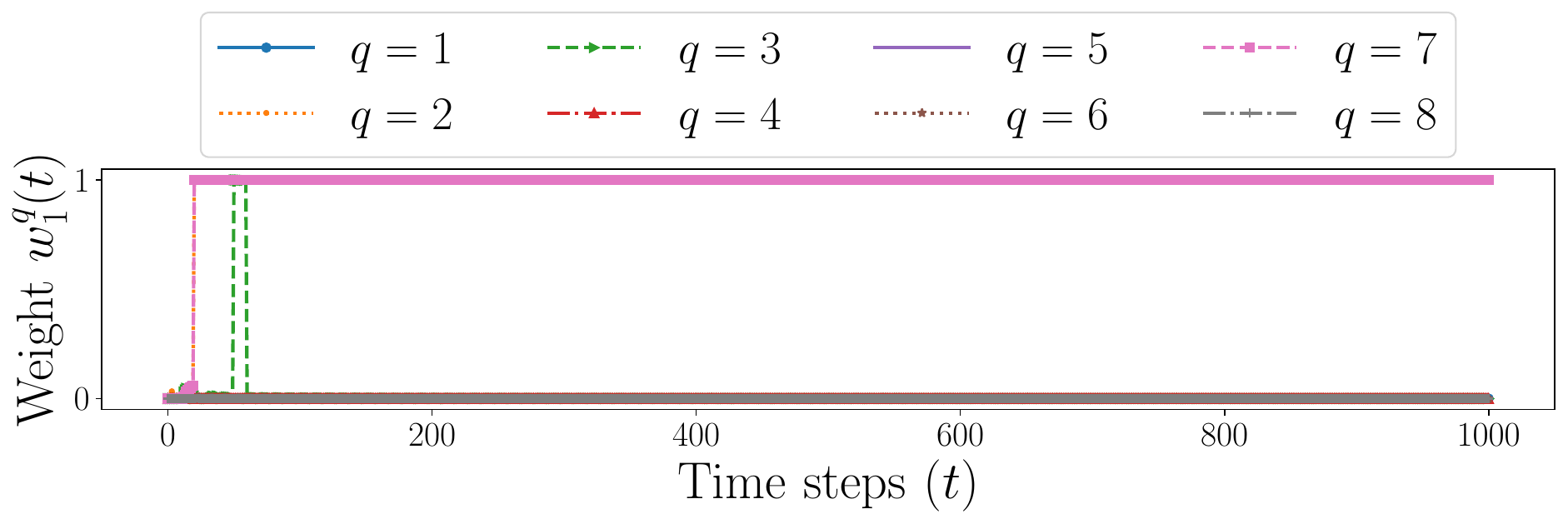} \\
		\hspace{-2ex}\includegraphics[scale=0.28]{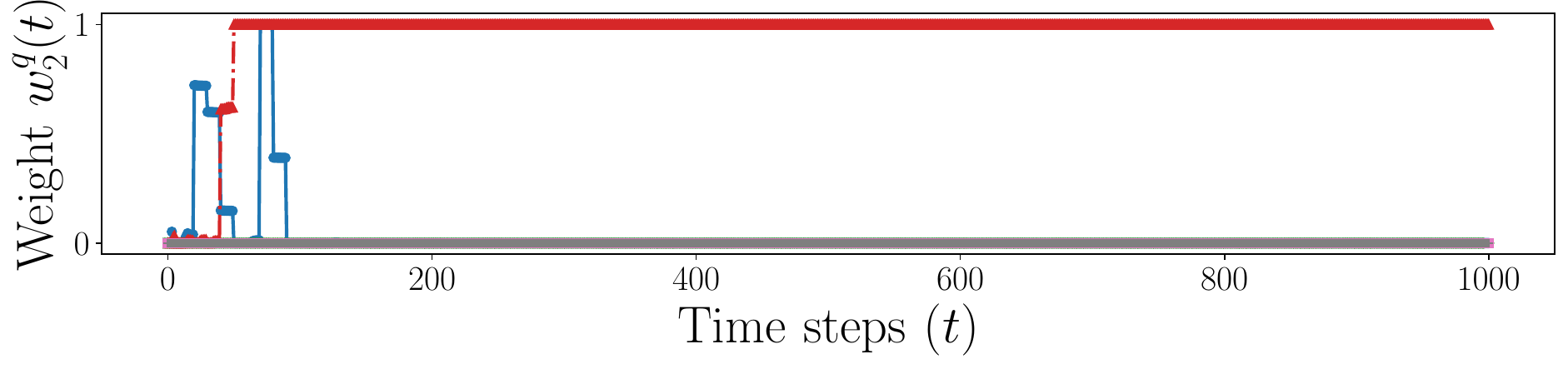} \\
		\hspace{-2ex}\includegraphics[scale=0.28]{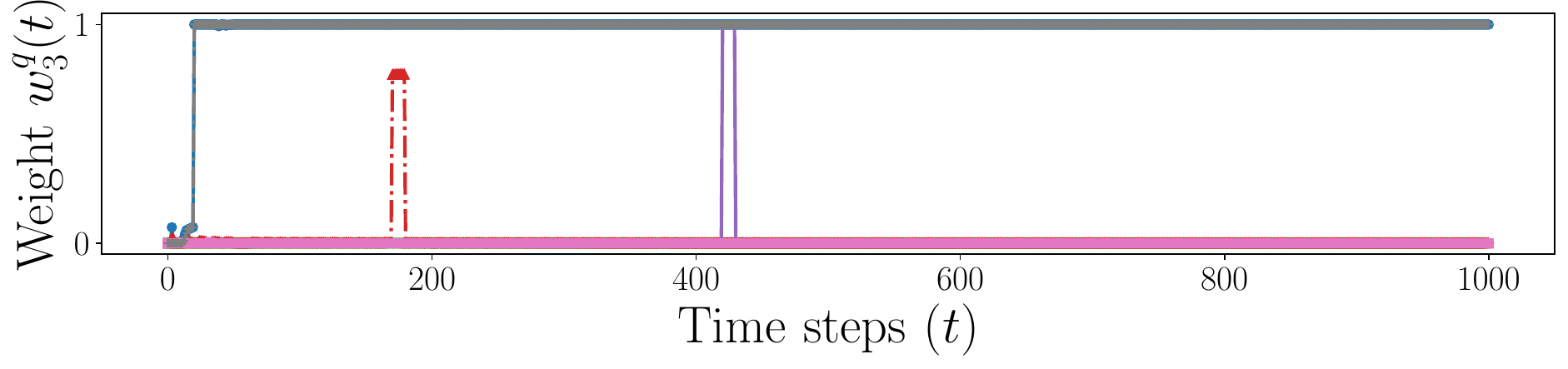} \\
		\hspace{-2ex}\includegraphics[scale=0.28]{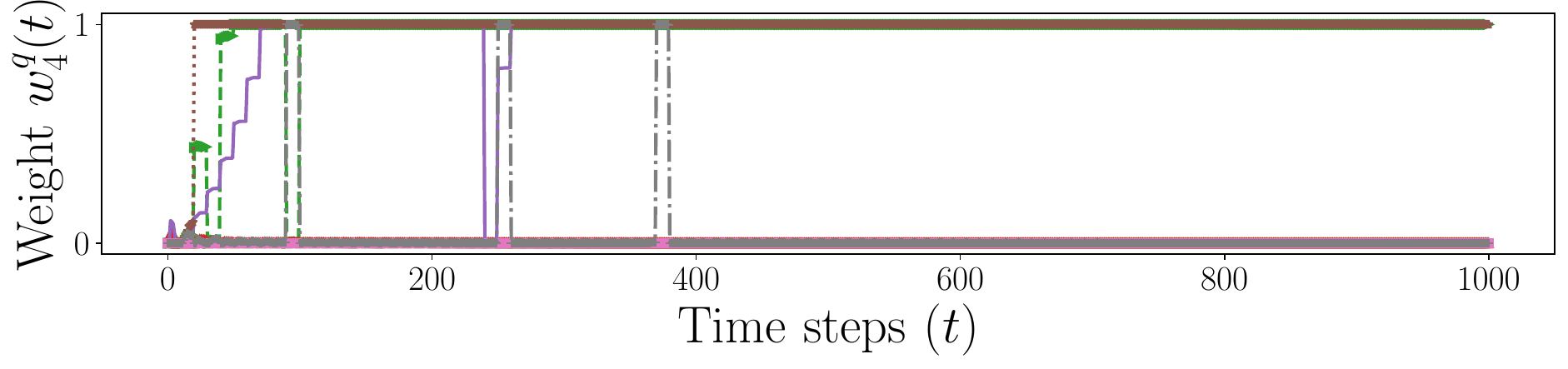}
	\end{tabular}
\end{center}
\caption{d-\dynacr\, for unknown rewards with time-varying step-size
  using~\labelcref{eq:dist_pbrag} on cyclic communication graph. The plots
  share a common
  legend.} %(a)-(d) Weight evolution of agents 1-4 respectively.}
	\label{fig:dist_comm}
\end{figure}

\section{Conclusion and Future Work} \label{sec:conclude}

In this paper, we presented a game theoretic formulation of an optimal task
allocation problem for a group of agents. By allowing
agents to assign weights between \emph{zero} and \emph{one} for each
task, we relaxed the combinatorial nature of the problem. This led to
a partition and weight game, whose NE formed a superset of the optimal
task partition. Then, we provided a distributed best-response projected gradient ascent by which convergence to the NE of the weight game
was guaranteed. %We verified our results using simulations.

Future work will consider constraints on number of tasks for each
agent, %introducing precedence order on the tasks 
and generalizing the setup to continuous space of tasks and classes of
tasks.

%\vspace{2ex}
\bibliographystyle{IEEEtran}
%\bibliographystyle{plain}
%\bibliography{ref.bib}
%\bibliography{../bib/alias.bib,../bib/SMD-add.bib,../bib/SM.bib,../bib/JC.bib,../bib/FB.bib}
\bibliography{alias.bib,SMD-add.bib,SM.bib,JC.bib,FB.bib}
%\bibliography{alias,SMD-add,SM,JC}

\end{document}